\documentclass{nMCM2e}

\usepackage{amsmath,amssymb}
\usepackage{enumerate}
\usepackage{pgfplots}
\usepackage{bm}
\usepackage{lmodern}
\usepackage{fix-cm}
\usepackage{graphicx}
\usepackage{mydefs}
\usepackage{algorithmic}
\usepackage{algorithm}
\pgfplotsset{compat=1.11}

\newcommand{\ltp}{{\mathcal{G}}}
\newcommand{\ltph}{{\mathcal{H}}}

\newcommand{\ltig}{{{G}}}
\newcommand{\ltih}{{{H}}}
\def\serkan#1{\textcolor{black}{{#1}}}
\def\caleb#1{\textcolor{black}{{#1}}}
\def\chris#1{\textcolor{black}{{#1}}}
\def\serkanlast#1{\textcolor{black}{{#1}}}

\newcommand{\sys}[4]{\ensuremath{\left[ \begin{array}{c|c} #1 & #2 \\ \hline #3 & #4 \end{array} \right ]}}

\usepackage{amsthm}

\theoremstyle{plain}
\newtheorem{theorem}{Theorem}[section]
\newtheorem{corollary}[theorem]{Corollary}
\newtheorem{lemma}[theorem]{Lemma}

\theoremstyle{definition}
\newtheorem{definition}[theorem]{Definition}

\theoremstyle{remark}

\begin{document}

\title{Linear time-periodic dynamical systems: An $\mathsf{H}_2$ analysis and a model reduction framework}

\author{
\name{C.C. Magruder\textsuperscript{a}$^{\ast}$\thanks{$^\ast$Corresponding author. Email: cm47@rice.edu},
 S. Gugercin\textsuperscript{b} and C.A. Beattie\textsuperscript{b}}
\affil{
	\textsuperscript{a}Department of Computational and Applied Mathematics, Rice University,\\ 6100 Main St. - MS 134, Houston, TX, 77005-1892, USA;\\
	\textsuperscript{b}Department of Mathematics, Virginia Tech,\\ 460 McBryde Hall, 225 Stanger Street
Blacksburg, VA 24061-0123 }
\received{Original Submission June 2016, Revision received June 2017}
}

\maketitle

\begin{abstract}
%
%
Linear time-periodic (\textsf{LTP}) dynamical systems frequently appear in the modeling of phenomena related to fluid dynamics, electronic circuits, and structural mechanics via linearization centered around known periodic orbits of nonlinear models. Such \textsf{LTP} systems can reach orders that make repeated simulation or other necessary analysis prohibitive, motivating the need for model reduction.

We develop here an algorithmic framework for constructing reduced models that \caleb{retains} the linear time-periodic structure of the original \textsf{LTP} system. Our approach generalizes optimal approaches that have been established previously for linear time-invariant (\textsf{LTI}) model reduction problems.  We employ an extension of the usual $\H2$ Hardy space defined for the \textsf{LTI} setting to time-periodic systems and within this broader framework develop an a posteriori error bound expressible in terms of related \textsf{LTI} systems.  Optimization of this bound motivates our algorithm. We illustrate the success of our method on two numerical examples.

\end{abstract}


\begin{classcode}
	37M05, 37M99, 93A15, 65K99, 93C05
\end{classcode}

\section{Introduction}
	We consider linear continuous-time-periodic (\textsf{\textsf{LTP}}) single-input/single-output (\textsf{SISO}) dynamical systems realized through a state-space representation,
	\begin{align}
		\label{eq:FOM}
		\ltp:\left \{ \begin{aligned}
			\dot{\x}(t) &= \A(t)\x(t) + \b(t)u(t)\\
			y(t) &= \c^T(t)\x(t)
		\end{aligned} \right .
		\left(\mbox{abbreviated }
		\ltp = \left[\begin{array}{c|c}\A(t) & \b(t)\\ \hline \c^T(t) &  \end{array}\right]\right),
	\end{align}
	where $\A(t) \in \R^{n \times n}$ and $\b(t),\c(t)\in\R^{n}$ are $T$-periodic, 
	with $\A(t) = \A(t+T)$, $\b(t) = \b(t+T)$ and $\c(t) = \c(t+T)$ for some fixed $T>0$ and we assume that $\x(0)=0$.
	Let $\mathsf L_2$ denote the vector space of real-valued square integrable functions (having ``bounded energy").
	\begin{align*}
		 \mathsf L_2 := \left \{ u : \int_0^\infty |u(t)|^2 dt < \infty \right \}.
	\end{align*}
	We assume that the system $\ltp$ in (\ref{eq:FOM}) is causal and is a bounded, linear mapping from $\mathsf L_2$-inputs to $\mathsf L_2$-outputs, $\ltp: \mathsf L_2 \to \mathsf L_2$.

	For any given order $r\ll n$, our goal is to find a \emph{reduced-order model},
	\begin{align}
		\label{eq:ROM}
		\widetilde \ltp:\left \{ \begin{aligned}
			\dot{\widetilde \x}(t) &= \widetilde{\A}(t)\widetilde{\x}(t) + \widetilde{\b}(t)u(t)\\
			\widetilde y(t) &= \widetilde{\c}^T(t)\widetilde{\x}(t)
		\end{aligned} \right .
\left(\mbox{abbreviated }		
\tilde \ltp = \left [\begin{array}{c|c}\tilde \A(t) & \tilde \b(t)\\ \hline \tilde \c^T(t) & \mathbf 0\end{array}\right ]\right),
	\end{align}
	\noindent where $\widetilde{\x}(0)=0$ and with $\widetilde{\A}(t) \in \R^{r \times r}$ and $\widetilde{\b}(t),\widetilde{\c}(t)\in\R^{r}$ also $T$-periodic and chosen in such a way so that $\widetilde y(t) \approx y(t)$ over a wide class of inputs $u(t)$.	

	\serkan{Starting with a Floquet transformation of the  \textsf{\textsf{LTP}}   system \eqref{eq:FOM},} our approach to this problem involves  conversion of the \textsf{LTP} model reduction problem into a closely related 
	 linear time-invariant (\textsf{LTI}) \caleb{multiple-input/multiple-output} (\textsf{MIMO}) 
		 model reduction problem,  connecting the model reduction error in the original \textsf{\textsf{LTP}} setting
		 directly to \chris{a} related \textsf{LTI} model reduction error.  In particular, we are able to provide \emph{a posteriori}  bounds for the error between the full- and reduced-order \textsf{\textsf{LTP}}          
		 systems directly in terms of the error in \chris{an} \textsf{LTI} \textsf{MIMO} model reduction counterpart.
		 This in turn invites the use of the Iterative Rational Krylov Algorithm (\textsf{IRKA}) \cite{Gugercin:2008bc}
		 as a way of minimizing the error bound we establish. 		 
	\chris{We find that this approach generates superior performance over other applicable and commonly used model reduction techniques for the examples considered in this paper.}
	
\section{Background}

\subsection{Previous work on the analysis and reduction of \textsf{\textsf{LTP}} systems}

	A significant body of literature has developed focussing on model reduction of linear time-invariant (\textsf{LTI}) systems \cite{Antoulas:2009tb,BenGS15,antoulas2001survey}.
	Analogous developments for continuous-time \textsf{\textsf{LTP}} systems have been far more limited and usually appear as a special case of the more general problem class of model reduction for general linear time-varying systems where computationally effective strategies for large scale problems have not yet emerged (see, e.g., \cite{sandberg2004balanced,sandberg2006case,lang2015towards}).  Available strategies for this problem have focussed on extensions of balanced truncation to general linear time-varying systems. The computational challenges are formidable even for modest order and are typified by the need to solve two large-scale Lyapunov differential inequalities \caleb{\cite{sandberg2004balanced}}.  Notably, the recent work  \cite{lang2015towards} is a step towards addressing this computational challenge.   By way of contrast, it is worth noting that for the case of linear \emph{discrete-time} periodic systems, strategies built upon balanced truncation have been somewhat more successful, producing strategies that still are computationally demanding but that could remain tractable for problems with modest order (see, e.g., \cite{varga1999balancing,varga2000balanced,farhood2005model,Chahlaoui2005,benner2014low}).  
Throughout this work, our focus is on continuous-time \textsf{\textsf{LTP}} systems.

Aside from model reduction, there is a significant body of literature that has developed on \textsf{\textsf{LTP}} systems focussing on a variety of other important topics including control, spectral analysis, and harmonic response, among others. 
Wereley and Hall \cite{Wereley:1991us, Wereley:1990wj} develop a Fourier analysis of state-space systems with periodic matrices via the harmonic balance method.
They represent the \textsf{\textsf{LTP}} system as generalization of a frequency response operator called the \emph{Harmonic Transfer Function} which can be viewed as an operator on a family of functions called \emph{exponentially modulated periodic inputs}.
%
%
The operator is infinite-dimensional \chris{and} consists of a countable number of time-invariant \textsf{LTI} components.
		
Sandberg et al.\ \cite{Sandberg:2006tk,Sandberg:2004vh} expand the impulse response of \textsf{\textsf{LTP}} systems via Fourier series expansion.
The frequency response of the system is understood as a countable collection of infinite-dimensional \textsf{LTI} transfer functions called the \emph{Floquet-Fourier representation}, making analytic expressions of frequency responses considerably easier to represent .
		
Zhou and Hagiwara \cite{Zhou:2005vs,Zhou:2004vh,Zhou:2002tp,Zhou:2002wf,Zhou:2002tf} address dynamical system norm computation using truncations of the Harmonic Transfer Function.
Convergence of the truncations is proven.
Additionally, the authors discuss the spectral characteristics of {\textsf{LTP}} systems as input/output operators.

\subsection{Linear Time-Invariant Dynamical Systems}

We review briefly the basic features of linear time-invariant dynamical systems that are relevant to our approach and that will serve to establish our notation.  In that context, we consider \caleb{\textsf{MIMO}} linear dynamical systems given in the following state-space form:
\begin{align} \label{eq:ltisys}
				\ltig:\left \{ \begin{aligned}
					\dot{\x}(t) &= \A\x(t) + \mathbf B\mathbf u(t)\\
					\mathbf y(t) &= \mathbf C\x(t)
				\end{aligned} \right .
				\qquad 
			\end{align}
where   $\mathbf u(t)\in\R^{\serkan{n_u}}$, $\mathbf y(t) \in\R^{\serkan{n_y}}$, and $\mathbf x(t)\in\R^n$ are respectively, the \emph{inputs}, \emph{outputs}, and \emph{states} of $\ltig$ so that  $\mathbf A\in\R^{n \times n}$, $\mathbf B\in\R^{n \times \serkan{n_u}}$ and $\mathbf C\in\R^{\serkan{n_y} \times n}$.  
Let	 $\hat{\mathbf{Y}}(s)$ and  $\hat{\mathbf{U}}(s)$ denote the Laplace transforms of $\mathbf{y}(t)$ and $\mathbf{u}(t)$, respectively.
Then by taking the Laplace transform of (\ref{eq:ltisys}), we obtain the transfer function $\mathbf{G}(s)$ of the system $\ltig$, where $\hat{\mathbf{Y}}(s) = \mathbf{G}(s) \hat{\mathbf{U}}(s) $ and
\begin{equation}
	\label{eq:ltitf}
	\mathbf{G}(s) = \mathbf{C} [s\I-\A]^{-1}\mathbf{B} 
\end{equation}
is the \emph{transfer function} of the associated dynamical system $\ltig$ in (\ref{eq:ltisys}).  $ \mathbf{G}(s)$ is a matrix-valued function
 $\mathbf{G}:\serkan{\C \to \C^{\serkan{n_u}}\times\C^{\serkan{n_y}}}$ with entries consisting of proper rational functions with poles at the eigenvalues of $\A$.

\subsubsection{Projection-Based Model Reduction}
When the state-space \caleb{dimension of} $\ltig$ is large,  it may be useful 
to replace  the original model (\ref{eq:ltisys}) with a reduced model that has much smaller state space dimension 
yet still is able to replicate significant features of the original input/output dynamics. Toward this end,  
 we seek a reduced model 
\begin{align}  \label{eq:ltisysred}
		\widetilde \ltig:\left \{ \begin{aligned}
			\dot{\widetilde \x}(t) &= \widetilde{\A}\widetilde{\x}(t) + \widetilde{\mathbf{B}}\mathbf{u}(t)\\
			\widetilde{\mathbf{y}}(t) &=  \widetilde{\mathbf{C}}\widetilde{\x}(t)
		\end{aligned} \right .
	\end{align}
where  $\widetilde{\A} \in \R^{r \times r}$, $\widetilde{\bB}\in\R^{r \times \serkan{n_u}}$, and $\widetilde{\bC}\in\R^{\serkan{n_y} \times r}$ with $r \ll n$, and such that $\widetilde \by(t) \approx \by(t)$ for a large class of \caleb{inputs $\mathbf{u}(t)$}.
In order to make the quality of this approximation largely independent of input, it becomes necessary that the transfer function of~(\ref{eq:ltisysred}) given by 
\begin{equation}
\widetilde{\mathbf{G}}(s)=\widetilde{\mathbf{C}}(s\widetilde{\I} - \widetilde{\A})^{-1}\widetilde{\B}
\end{equation}
should approximate $\mathbf{G}(s)$ well, to the extent that the error
$\mathbf{G}(s)-\widetilde{\mathbf{G}}(s)$ is small with respect to natural metrics that we discuss below. 

The most common way to obtain the reduced-order models as in (\ref{eq:ltisysred}) is via projection:  Construct  two  matrices  $\mathbf{V} \in \R^{n \times r}$ and $\mathbf{W} \in \R^{n \times r}$ such that $\mathbf{W}^T\mathbf{V} = \mathbf{I}_r$.
Then, approximate the full-order state $\mathbf{x}(t)$ by $\mathbf{V}\widetilde{\mathbf{x}}(t)$, 
and enforce the Petrov-Galerkin condition, forcing orthogonality of the residual dynamics to the range of $\mathbf{W}$:
$$
\mathbf{W}^{T}\left( \mathbf{V}\dot{\widetilde{\mathbf{x}}} (t)-
\A\mathbf{V}\widetilde{\mathbf{x}}(t)-\mathbf{B}\,\mathbf{u}(t)\right)=\mathbf{0}, \qquad
\quad \widetilde{\mathbf{y}}(t) = \mathbf{C}\mathbf{V} \widetilde{\mathbf{x}}(t).
$$
Then the reduced model state-space matrices become 
\begin{equation}
	\widetilde{\A} = \W^T\A\V,  ~\widetilde{\mathbf{B}} = \W^T \mathbf{B}, ~~~\mbox{and}~~~\widetilde{\mathbf{C}} = \mathbf{C}\V
\end{equation}
Obviously the choice of $\V$ and $\W$ will determine the accuracy of the reduced model approximation. 

\subsubsection{The $\Hardy_2$ Inner-product Space}

Let  $\Hardy_{2}$ denote the set of  $n_y\times \serkan{n_u}$ matrix-valued functions \caleb{$\bH(s)$ analytic} in the open right half-plane  such that $\sup_{x>0}\int_{\serkan{-\infty}}^{\infty} \|\bH(x+\i y)\|_F^2\, dy < \infty$. 
\serkan{The $\Hardy_{2}$ space}
is a   Hilbert space
 endowed with the  inner product 
	\serkan{
		 \begin{align} 
			\langle \mathbf G,\mathbf H\rangle_{\H2}
			= \frac{1}{2\pi} \int_{-\infty}^\infty \mbox{\normalfont\textsf{trace}}\!\left ( \overline{\mathbf G(\i\omega)} \mathbf H^T(\i\omega)\right )d\omega
			= \frac{1}{2\pi} \int_{-\infty}^\infty\mbox{\normalfont\textsf{trace}}\!\left ( \mathbf G(-\i\omega) \mathbf H^T(\i\omega)\right )d\omega
	\label{eqn:h2innerlt}
		\end{align}}
		with the associated norm
		\begin{align}  \label{eq:h2normlti}
			\left\|\mathbf G\right\|_{\H2}
			= \left (\frac{1}{2\pi} \int_{-\infty}^\infty \|\mathbf G(\i\omega)\|_F^2 d\omega \right )^{1/2}
			=  \left (\int_0^\infty \mbox{\normalfont\textsf{trace}}(\bm{g}^*(t) \bm{g}(t))dt \right)^{1/2}
		\end{align}
	\noindent where $\bm{g}(t)\in\R^{\serkan{n_y}\times \serkan{n_u}}$ is the matrix-valued impulse response of the \textsf{MIMO} dynamical system, $\ltig$.
	The $\H2$ norm is of particular interest because it bounds the dynamical system output for bounded-energy inputs as,
		\begin{align*}
			\left \|\by\right\|_{\mathsf L_\infty} \leq \left\|\bG\right\|_{\H2}\left\|\bu\right\|_{\mathsf L_2}.
		\end{align*}		
		 The next result, used frequently throughout this paper, provides  an alternative representation of the $\Hardy_2$ inner product using the residue calculus.

		\begin{theorem}{\cite[Lemma 2.4, p.615]{Gugercin:2008bc}}
    Suppose that $\bG(s)$ and $\bH(s)$ are stable (poles contained in the open left 
halfplane) and suppose that  $\bH(s)$ has poles at 
$\mu_1,\,\mu_2,\,\ldots\,\mu_n$.  Then
\begin{equation} \label{h2ip_res}
    \left\langle \bG,\ \bH \right\rangle_{\Hardy_2} =  \sum_{k=1}^{n} 
    \mbox{\normalfont\textsf{res}}[\mbox{\normalfont\textsf{Tr}}\!
    \left( \overline{\bG}(-s) \bH(s)^T\right),\mu_k].
\end{equation}
    In particular, 
     if $\bH(s)$ has only simple or semi-simple poles at
     $\mu_1,\,\mu_2,\,\ldots\,\mu_n$, then
$\bH(s)=\sum_{i=1}^n\frac{1}{s-\mu_i}\bfsfc_i\bfsfb_i^T$, \serkan{where $\bfsfc_i\bfsfb_i^T$ is the residue of $\bH(s)$ at $s=\mu_i$},
and
\begin{align*}
    \langle \bG,\,\bH \rangle_{\Hardy_2} =  \sum_{k=1}^{n} 
   \bfsfc_k^T\overline{\bG}(-\mu_k)\bfsfb_k~~\mbox{and}~~
    \left\| \bH \right\|_{\Hardy_2} =  \left(\sum_{k=1}^{n} 
   \bfsfc_k^T\overline{\bH}(-\mu_k)\bfsfb_k \right)^{1/2}.
\end{align*}
		\end{theorem}

\serkan{\subsubsection{Optimal $\H2$ Approximation}  \label{sec:irkah2}}
We seek  reduced models for \textsf{LTP} systems that are accurate with respect to a metric that is equivalent to the  $\Hardy_2$ norm for \textsf{LTI} systems (which are special cases of \textsf{LTP} systems).  Indeed, for the case of  \textsf{LTI} systems, we are able to do this optimally, finding a degree-$r$ \textsf{LTI} reduced model with transfer function $\bH_r(s)$ that minimizes (at least locally) the $\Hardy_2$ error norm $\|\bH - \widetilde{\bH}_r\|_{\Hardy_2}$  over all degree-$r$ reduced models with transfer functions $\widetilde{\bH}(s)$.
The optimal \textsf{LTI} $\Hardy_2$ approximant satisfies certain Hermite tangential interpolation conditions; for details, see, e.g., \cite{Ant2010imr,Gugercin:2008bc}.  Significant for the present work is the availability of a numerically effective algorithm 
with which one may construct (local) $\Hardy_2$-optimal approximants, namely the \emph{Iterative Rational Krylov Algorithm} (\textsf{IRKA}) of  \cite{Gugercin:2008bc}.  
\serkan{\textsf{IRKA} is an iterative algorithm;  it typically converges quite rapidly
though its speed of convergence may slow as the number of inputs and outputs grows. 
Beattie and Gugercin in \cite{beattie2012realization} 
developed a modified version of \textsf{IRKA}, in order to address this slowing of convergence and improved  the performance of IRKA significantly for MIMO problems; see \cite{beattie2012realization} for more details. 
  Convergence of \textsf{IRKA} is guaranteed for special cases \cite{flagg2012convergence} even though
there are known, but rare, cases where convergence may fail \cite{Gugercin:2008bc,flagg2012convergence}. Upon convergence, the resulting reduced model is guaranteed to be a local $\mathcal{H}_2$-minimizer. 
 In \cite{beattie2009trm}, Beattie and Gugercin developed a trust-region framework for \textsf{IRKA}, which is \emph{globally convergent} to a local $\mathcal{H}_2$-minimizer.  The added guarantees this approach provides seem unnecessary for most problems at hand and the original formulation of \textsf{IRKA}  has been successfully applied to large-scale problems in various application settings in order to provide (locally) optimal reduced models; see, e.g., \cite{KRXC08,borggaard2012model,Gugercin:2008bc}.}
We will employ \textsf{IRKA} as a critical step in the model reduction framework for \textsf{\textsf{LTP}} systems that we propose here.
   
\subsection{Linear Time-Periodic Dynamical Systems}
	\label{chap:LTP}
\caleb{We view the time-periodic system \eqref{eq:FOM} as a causal linear map $\ltp: \L2 \to \L2$, that maps an input signal $u\in \L2$ to an output signal $y \in \L2$.} This can be expressed via a convolution integral with an generalized impulse \caleb{response $g(t,\tau)$ which,} for $t \geq \tau$, gives the response of the system \eqref{eq:FOM} to an impulsive input at $t=\tau$.  Thus,
	\begin{align}
		\label{eqn:kernel}
		y(t) &= \ltp u = \int_{0}^t g(t,\tau)u(\tau)d\tau,
	\end{align}

Unlike \textsf{LTI} systems, the impulse response of a time-varying system depends on the time of the \caleb{impulse $\tau$ and} not simply on the time elapsed since the impulse was applied, $t-\tau$. 
Since the state-space parameters are $T$-periodic, so too are the impulse \caleb{responses.
That is,}
\begin{align*}
	g(t+T,\tau+T) = g(t,\tau),\quad \mbox{for all} \quad t\geq \tau.
\end{align*}

\noindent \caleb{In this setting, causality implies that $g(t,\tau)=0$ for $t < \tau$.}

\subsubsection{Frequency Coupling}
	An important characteristic of stable \textsf{LTI} systems is that their steady state response to a single-frequency sinusoid is another sinusoid of the same frequency. That is, if $G(s)$ is the transfer function of an \textsf{LTI} \caleb{system $\ltp$, then}
	\begin{align*}
		u(t) = \sin(\omega t) \qquad \overset{y = \ltp u}{\Longrightarrow} \qquad y(t) \to |G(\i\omega)|\sin[\omega t + \mbox{arg}(G(\i\omega))].
	\end{align*}
	Note that we write ``$y(t) \to $" to describe the steady-state behavior of $y(t)$ after the transient response has decayed.

	By way of contrast, \textsf{LTP} systems produce a countable number of harmonics of the input frequency. 
\serkan{Let $r = t-\tau$. Since $g(t,t-r)$ is $T$-periodic in $t$, the impulse response can be expanded into a Fourier series in \caleb{$r$,}
	\begin{align*}
		g(t,t-r) = \sum_{k = -\infty}^{\infty} { g_k(r)e^{\i k\omega_0 t}, \quad\mbox{where} \quad
		g_k(r) }= \frac{1}{T}\int^T_0 e^{-\i k\omega_0 t}g(t,t-r)dt.
	\end{align*}
	We call  $\{g_k\}$ the subsystems of $\ltp$. } Now,
	let $u(t) = \sin(\omega t)$ and $\displaystyle\omega_0 = \frac{2\pi}{T}$ be the fundamental frequency of the \textsf{LTP} system $\ltp$. Then,
	\begin{align*}
 		u(t) = \sin(\omega t)
		\quad \overset{y = \ltp u}{\Longrightarrow} \quad
		y(t) \to \sum_{k\in\Z}|g_k(\i\omega)|\sin[(\omega + k\omega_0) t + \mbox{arg} (g_k(\i\omega))].
	\end{align*}
	For a single frequency input, an \textsf{LTP} system will have output frequencies \caleb{$\omega_k = \omega+k\omega_0$}, where $k \in \mathbb{Z}$; \caleb{see Sandberg \cite{Sandberg:2006tk}}.
%

\subsubsection{Floquet Transformations}
		The Floquet transformation is a time-dependent change-of-variable that transforms the time-periodic differential equation
		\begin{align}
			\label{homogenousLTP}
			\dot{\x} = \A(t)\x,\qquad \A(t) = \A(t+T)
		\end{align}
		to an equivalent system of differential equations with constant coefficients, $\dot{\mathbf z} = \Q \mathbf z(t)$ via a periodic, time-dependent change of variables, $\mathbf P(t)\mathbf z(t) = \x(t)$; see, e.g., \cite{grimshaw1991nonlinear} for details.
		In particular, we have
		\begin{theorem}[Floquet's Theorem]  \label{thm:floquet}
		Let $\A(t) \in \mathbb{R}^{n \times n}$ be  continuous 
and periodic with 	period $T$ 	for $-\infty<t<\infty$
			Then any fundamental matrix $\mathbf X(t)$ of the differential equation				$\dot{\x} = \A(t)\x$
	 has a representation of the form
			\begin{align*}
				\mathbf X(t) = \mathbf P(t)e^{\Q t}, \qquad \mbox{where~~}
			\mathbf P(t) \in \mathbb{R}^{n \times n} \mbox{~with~}
				\mathbf P(t+T) = \mathbf P(t),
			\end{align*}
			and $\Q \in \mathbb{R}^{n \times n}$ is a constant matrix.
		\end{theorem}
		
The constant matrix, $\Q$ is defined so that $\Q = \frac1T\mathsf{log}(\mathbf M(T))$ where \serkan{the matrix} $\mathbf M(T)=\mathbf X(T)\mathbf X(0)^{-1}$
is the \emph{monodromy matrix} associated with (\ref{homogenousLTP}).
Note that $\Q$ is stable precisely when $\mathbf M(T)$ is a contraction.  The $T$-periodic matrix $\P(t)$ is defined as
$\mathbf P(t)=\mathbf X(t) \mathbf X(0)^{-1} e^{-\Q t}$ for $t\geq 0$. 

\caleb{Given an \textsf{LTP} system, $\ltp$, a Floquet transformation using the periodic change of variable $\mathbf z(t) = \P^{-1}(t) \x(t)$ from Theorem \ref{thm:floquet}, can be performed} so that 
	\begin{equation} \label{postFloquet}
		\ltp:\left\{\begin{aligned}
			\dx(t) &= \A(t) \x(t) + \b(t) u(t)\\
		y(t) &= \c^T(t) \x(t)
	\end{aligned}\right. ~ \Longrightarrow ~
	\ltp:\left\{\begin{aligned}
		\dot{\mathbf z}(t) &= \Q \mathbf z(t) + \P^{-1}(t)\b(t) u(t)\\
		y(t) &= \c^T(t) \P(t) \mathbf z(t)
	\end{aligned}\right.
\end{equation}
The Floquet transformation can be costly to determine for large-scale \textsf{LTP} systems. Methods have been developed recently to make the transformation computationally tractable for modest order; see, e.g., \cite{Moore:2005un,Cai:2001wh} for a discussion.  The main computational cost in most cases will be the calculation of the monodromy matrix, which necessitates the solution of $n$ independent initial value problems associated with \eqref{homogenousLTP}. The remaining tasks scale with a complexity of $n^3$, so this can be feasible for modest orders of $n$.  \chris{In this work, we will not consider the practical difficulties associated with constructing this  transformation of the original \textsf{LTP} system; we will suppose that our \textsf{LTP} system is presented with a realization having the form \eqref{postFloquet} with constant $\Q$ and  $T$-periodic input and output maps. }

\section{An  $\H2$  Analysis for \textsf{LTP} systems}
We develop here for \textsf{LTP} systems the metrics and associated analysis that extend certain notions of $\H2$-approximation that are well-established for  \textsf{LTI}  systems. 
\subsection{The Periodic $\H2$ Inner Product}

Let  $\ltp$ and $\ltph$ be two \textsf{LTP} systems with fundamental frequency $\displaystyle\omega_0 = \frac{2\pi}{T}$. The $\H2$ inner product  $\langle \ltp,\ltph\rangle_\H2$ is defined in terms of their impulse \caleb{responses $g(t,\tau)$ and $h(t,\tau)$},
\begin{align} \label{eqn:h2inner}
			\langle \ltp,\ltph \rangle_{\H2} &= \frac{1}{T} \int_{t=0}^T \int_{r=0}^\infty \overline{g(t,t-r)}\, h(t,t-r) drdt.
			\end{align}
The next results express this inner product in terms of the kernel subsystems; this can be found as Corollary 1 in Sandberg et al.\ \cite{Sandberg:2004vh}.
Our formulation is framed in the frequency domain whereas the original formulation of
Sandberg et al.\ \cite{Sandberg:2004vh} places it in the time domain.
	\begin{theorem} 
		\label{thm:h2innerproduct}
		Let $\ltp$ and $\ltph$ be two \textsf{LTP} systems with subsystems $g_k(t)$ and $h_k(t)$, respectively. Then,
		\begin{align} \label{eq:sandberhh2inner}
			\langle \ltp,\ltph\rangle_\H2 = \sum_{ k \in \Z} \langle \hat g_k, \hat h_k\rangle_{\H2}
		\end{align}
		where $\hat g_k(s)$ and $\hat h_k(s)$ denote the Laplace transformations of $g_k(t)$ and $h_k(t)$, respectively, and 
		$\langle \hat g_k, \hat h_k\rangle_{\H2}$ is the regular $\H2$-inner product for  \textsf{LTI} systems defined in  \eqref{eqn:h2innerlt}.
	\end{theorem}

	It follows immediately that 
	\begin{equation}  \label{eq:LTIH2sub}
	\left\|\ltp\right\|_\H2^2 = \sum_{k \in \Z} \left\|\hat g_k\right\|^2_{\H2}.
	\end{equation}
	 This expression for the $\H2$ norm will allow us to develop a simple test for the boundedness of the $\H2$ norm of an \textsf{LTP} system 
	 $\ltp$.  We adapt the following (standard) definition to our context: 
	 \begin{definition}  For $\alpha>0$, $\mathsf{Lip}(\alpha)$ denotes the set of continuous periodic functions $f(t)$ (say, with period T and fundamental frequency $\omega_0 = \frac{2\pi}{T}$) such that  for some finite $M>0$, $|f(t_1)-f(t_2)|\leq M|t_1-t_2|^{\alpha}$ uniformly for all $t_1,t_2\in [0,T]$. 
	 \end{definition}
\noindent
Note that $\mathsf{Lip}(\alpha)$ with $\alpha>1$ consists only of constant \caleb{functions and} $\mathsf{Lip}(1)$ consists of Lipschitz continuous
periodic functions which in turn will be contained within $\mathsf{Lip}(\alpha)$ for any $\alpha\in(0,1)$. 

	 Our main result related to the $\H2$ analysis of \textsf{LTP} systems makes use of the following two \serkan{lemmata:}
\begin{lemma}{(\cite[Theorem 1, p.176]{borwein2005weighted})}.
	\label{thm:borwein}
Let $\ell_1$ and $\ell_2$ denote the usual normed spaces of absolutely summable and square summable (scalar) sequences, respectively.
If $\serkan{\{a_k\}}\in \ell_1$ and $\serkan{\{x_k\}}\in \ell_1$ then the convolution $y_k = \sum_\ell a_{k-\ell}x_\ell$ is absolutely convergent;
	$\serkan{\{y_k\}}\in \ell_1$.   If $\serkan{\{a_k\}}\in \ell_1$ and $\serkan{\{x_k\}}\in \ell_2$ then the convolution $y_k = \sum_\ell a_{k-\ell}x_\ell$ is an  $\ell_2$ sequence;	$\serkan{\{y_k\}}\in \ell_2$.
\end{lemma}

\begin{lemma}
	\label{lemma:fourierpwsmooth} 
	Let the vector-valued function $\x(t)\in\R^n$ be $T$-periodic with components in $\mathsf{Lip}(\alpha)$ with $\alpha>\frac12$.
	Then 
	\begin{enumerate}
	\item the Fourier \caleb{expansion $\x(t) = \sum_{k\in\Z} \x_k e^{jk\omega_0 t}$} is absolutely convergent and $\{\|\x_k\|_2\}\in\ell_1$ 
	(Bernstein's Theorem, \cite[Theorem VI.3.1]{zygmund2002trigonometric}).
	\item the Fourier coefficients $\{\x_k\}$ satisfy  $\sqrt{k}|\x_k|\rightarrow 0$ (componentwise) as $k\rightarrow \infty$ (\cite[II.4]{zygmund2002trigonometric})
	\end{enumerate}
\end{lemma}

	The following theorem can be understood as a special case of Lemma 1 of Sandberg et al.\ \cite{Sandberg:2006tk}, which pertains to a more general class of impulse response functions. The proof provided here is instead a concise frequency domain-based proof for dynamical systems with state space representations.
		
	\begin{theorem}
		Given system $\ltp = \sys{\A(t)}{\b(t)}{\mathbf c^T(t)}{}$ with $\A(t)$, $\b(t)$ and $\mathbf c(t)$ having components that are 
		T-periodic and in $\mathsf{Lip}(\alpha)$ with $\alpha>\frac12$.  If the Floquet-transformed state space matrix $\Q$ is Hurwitz, then $\left\|\ltp\right\|_{\H2}<\infty$.
	\end{theorem}

	\begin{proof}
		Without loss of generality, assume $\ltp$ is in the Floquet form. Let
			$$\c(t) = \sum_{k\in\Z} \c_k e^{-\i k\omega_0 t}~~\mbox{and}~~
			\b(t) = \sum_{k\in\Z} \b_k e^{-\i k\omega_0 t}$$ denote the Fourier expansions of $\c(t)$ and $\b(t)$, respectively.
		\caleb{Then from Sandberg (2006) \cite{Sandberg:2006tk} we know $\hat g_k(s) = \sum_{\ell\in\Z} \c_{k-\ell}^T[s_\ell\I - \Q]^{-1}\b_\ell$ with $s_\ell = s + \i \ell \omega_0$.} 
		\caleb{Additionally, $\left\|\ltp\right\|_{\H2}^2 = \displaystyle\sum_{ k\in\Z}\left\|\hat g_k\right\|_{\H2}^2$ from \eqref{eq:LTIH2sub}}.
		First, show that $\|\hat g_k\|_{\H2} < \infty$:
		\begin{align*}
			\left\|\hat g_k\right\|_{\H2} &\leq \sum_{\ell\in\Z}\|\c_{ i-\ell}^T [s_\ell\I - \Q]^{-1}\b_\ell\|_{\H2}
			\leq \sum_{\ell\in\Z}\|\c_{k-\ell}\|_2\|\b_\ell\|_2 \|[s_\ell\I - \Q]^{-1}\|_{\H2}\\
			&=\|[s\I - \Q]^{-1}\|_{\H2} \sum_{\ell\in\Z} \|\c_{k-\ell}\|_2\|\b_\ell\|_2,
		\end{align*}
where in the last step we used 
$ \|[s_\ell\I - \Q]^{-1}\|_{\H2} = \|[s \I - \Q]^{-1}\|_{\H2}$. 
\serkan{One can immediately observe that   
\begin{align*}
			\left \{\sum_{\ell\in\Z} \|\c_{k-\ell}\|_2\|\b_\ell\|_2\right \}_{k=-\infty}^\infty = \Big\{\|\c_k\|_2\Big\}_{k=-\infty}^\infty * \Big\{\|\b_k\|_2\Big\}_{k=-\infty}^\infty \in \ell_1.
		\end{align*}
		where ``\,*\," denotes the convolution operator.
		Note that $\{\|\c_k\|_2\},\{\|\b_k\|_2\}\in\ell_1$ by Lemma \ref{lemma:fourierpwsmooth}. Therefore, it follows 
		from Lemma \ref{thm:borwein} that the sequence $\left \{\sum_{\ell\in\Z} \|\c_{k-\ell}\|_2\|\b_\ell\|_2\right \}$ is an $\ell_1$ sequence  indexed in $k$.}
		Since $\Q$ is Hurwitz, then $\|[s\I - \Q]^{-1}\|_{\H2}<\infty$. Thus, $\{\|\hat g_k\|_{\H2}\}$ forms an $\ell_1$ sequence indexed in $k$. Since $\ell_1\subset \ell_2$, $\|\ltp\|_{\H2}^2 = \displaystyle\sum_{k\in\Z}\|\hat g_k\|_{\H2}^2 < \infty$.
	\end{proof}

\subsection{Pole-Residue Representation of $\H2$ Inner Product}

An important result from \textsf{LTI} system theory is the representation of 
$\H2$ inner products and $\H2$ norms in terms of poles and residues as shown in \eqref{h2ip_res}. For 
 \textsf{SISO~LTI} systems, $G$ and $H$, that are stable, and assuming simple poles, $\{\lambda_i\}_{i=1}^n$ for $\bG(s)$
and $\{\mu_i\}_{i=1}^n$ for for $\bH(s)$, these inner product/norm representations simplify to:
\begin{align} \label{eq:h2poleressiso}
	\langle G,H \rangle_{\H2} &= \sum_{j=1}^n  \bG(-\mu_j){\normalfont\textsf{res}}[\bH(s),\mu_j] = \sum_{j=1}^n 
	\bH(-\lambda_j){\normalfont\textsf{res}}[\bG(s),\lambda_j], \\
	\mbox{and}~~~
\| G\|_{\H2}^2 &=  \sum_{i=1}^n \bG(-\lambda_j){\normalfont\textsf{res}}[\bG(s),\lambda_j].\label{eq:h2normpoleressiso}
\end{align}
Even though this formulation of the inner product is rarely used for computation,
it has been used in deriving optimality conditions for $\H2$ approximation, see
\cite{Gugercin:2008bc,Ant2010imr}.
\chris{In what follows, we extend this result to  \textsf{SISO~LTP} systems.}

\chris{In the \textsf{LTP} setting, the meromorphic functions in question no longer have a finite number of poles.  Although
the pole residue formulation of the inner product provided above extends naturally in the way one would hope, 
the proof of this extension} is considerably more difficult as the next theorem \caleb{shows}.

	\begin{theorem}
		\label{thm:h2poleresinnerprod}
		Let $\ltp = \sys{\Q}{\widetilde{\b}(t)}{\widetilde{\c}^T(t)}{}$ and $\ltph = \sys{\Q}{\b(t)}{\c^T(t)}{}$ denote two \textsf{LTP} systems sharing state space matrix $\Q$ with $\widetilde{\b}(t)$, $\widetilde{\c}^T(t)$, $\b(t)$, $\c^T(t)$  having components that are 
		 $T$-periodic and in $\mathsf{Lip}(\alpha)$ with $\alpha>\frac12$.
		 Let $\hat g_k,\hat h_k\in\H2$ be the \caleb{$k$-th} subsystems of $\ltp$ and $\ltph$ respectively where
		\begin{align*}
			\hat g_k(s) = \sum_{\ell\in\Z}\widetilde\c^T_{ k-\ell}[s_\ell\I-\Q]^{-1}\widetilde\b_\ell,
			\qquad
			\hat h_k(s) = \sum_{\ell\in\Z}\c^T_{ k-\ell}[s_\ell\I-\Q]^{-1}\b_\ell.
		\end{align*}
		Assume that the \caleb{eigenvalues $\{\lambda_j\}_{j=1}^n$ of $\Q$ are} in the open left-half plane and that the fundamental frequency $\omega_0$ bounds the largest imaginary component of the spectrum of $\Q$.
		That is,
		\begin{align*}
			\max_j |\serkan{\mathrm{Im}}\{\lambda_j(\Q)\}| < \omega_0.
		\end{align*}
		If $\Q$ has simple poles, then $\hat g_k$ and $\hat h_k$ have poles at 
		$\caleb{\lambda_j^{(\ell)} = \lambda_j(\Q) - \i \ell \omega_0, \ \ell\in\Z}$
		and the inner product, $\langle \hat g_k, \hat h_k \rangle_{\H2}$ can be written as
		\begin{align}
			\label{eq:h2poleresinnerprod}
			\langle \hat g_k, \hat h_k\rangle_{\H2} = \sum_{\ell\in\Z} \sum_{j=1}^n \hat g_k(-\lambda_j^{(\ell)}){\normalfont\textsf{res}}[\hat h_k(s),\lambda_j^{(\ell)}].
		\end{align}
		The $\H2$-norm of $\hat g_k(s)$ is given by
		\begin{align}
			\label{eq:LTIH2sub_poleres}
			\| \hat g_k \|^2_{\H2} = \sum_{\ell\in\Z}\sum_{j=1}^n \hat g_k(-\lambda_j^{(\ell)}){\normalfont\textsf{res}}[ \hat g_k(s),\lambda_j^{(\ell)}].
		\end{align}
	\end{theorem}
			
\begin{proof}
	Fix $k$.
	We note that residue calculus immediately provides the pole residue formulation for subsystems with a finite number of poles.
	However, the subsystems $\hat g_k$ and $\hat h_k$ we consider have a countable number of poles that tesselate along the imaginary axis, \caleb{$\lambda_j^{(\ell)} = \lambda_j(\Q) -\i \ell \omega_0$ for $\ \ell\in\Z$}.
	
	First, consider the case where $\Q$ is scalar: $n=1$, $\mathbf Q = \lambda$ with $\textrm{Re}\{\lambda\} < 0$.
	Without loss of generality we assume $\omega_0 = 1$. 
	Then \caleb{$\lambda^{(\ell)} = \lambda - \ell\i$.}
	Define $\phi_\ell = c_{ k-\ell} b_\ell$ and $\widetilde \phi_\ell = \widetilde c_{ k-\ell} \widetilde b_\ell$.
	\serkan{Thus, we obtain}
	\begin{align*}
		\caleb{\hat h_k(s) = \sum_{\ell\in\Z} \frac{\phi_\ell}{s-(\lambda-\ell\i)}
		\quad \mbox{and} \quad
		\hat g_k(s) = \sum_{\ell\in\Z} \frac{\widetilde \phi_\ell}{s-(\lambda-\ell\i)}}
	\end{align*}
	\serkan{with $\{\phi_\ell\},\{\widetilde \phi_\ell\}\in\ell_1$.}
			
	Consider the rectangular \caleb{contour $\Gamma_M$ defined} by the vertices $(0,M+\frac{1}{2})$, $(0,-M-\frac{1}{2})$, $(-M,-M-\frac{1}{2})$, $(-M,M+\frac{1}{2})$, for integer $M > |\lambda|$, so that $\Gamma_M$ contains $2M+1$ poles.
	Define $\Gamma_M^{(2)}$, $\Gamma_M^{(2)}$, $\Gamma_M^{(3)}$, $\gamma_M$ to represent the the top, left, bottom, and right contours of the rectangle respectively (see Figure \ref{fig:contour}).
	Finally, write $\Gamma_M = \Gamma_M^{(1)} + \Gamma_M^{(2)} + \Gamma_M^{(3)} + \gamma_M$ to represent the entire contour.
			
	\begin{figure}[H]
		\centering
		\includegraphics[scale=1.0]{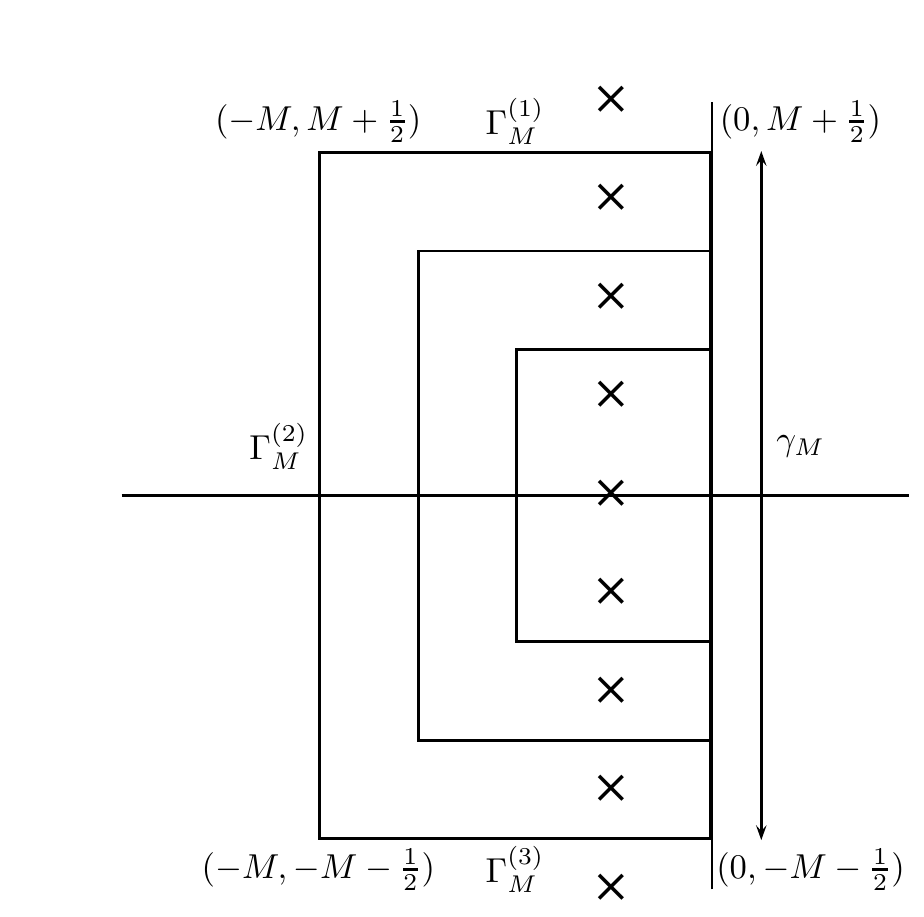}
		\caption{Contour integration path, $\Gamma_M$}
		\label{fig:contour}
	\end{figure}
			
	\noindent From residue calculus we have
	$\displaystyle \int_{\Gamma_M} \hat g_k(-s) \hat h_k(s) ds = \sum_{\ell=-M}^M \hat g_k(-\lambda^{(\ell)}){\normalfont\textsf{res}}[\hat h_k(s),\lambda^{(\ell)}]$.
	We first show that $\int_{\Gamma_M^{(2)}} \hat g_k(-s) \hat h_k(s) ds \to 0$ as $M\to\infty$.
	Observe  that $\Gamma_M^{(2)} = \{z = -M + \i y:-M-\frac{1}{2} \leq y \leq M + \frac{1}{2}\}$.
	Then,
	\begin{align*}
		|\hat g_k(-M + \i y) \hat h_k(M-\i y)|
		&\leq \left |\sum_{\ell\in\Z}\frac{\widetilde \phi_\ell}{-M-\lambda + \i(y\caleb{+\ell})}\right | \left | \sum_{\ell\in\Z} \frac{\phi_\ell}{M-\lambda+\i(-y\caleb{+\ell})}\right |\\
		&\leq \frac{1}{M^2-\lambda^2} \sum_{\ell\in\Z} \left|  \widetilde \phi_\ell\right |  \sum_{\ell\in\Z} \left | \phi_\ell\right |.
	\end{align*}
	\caleb{We know that if $f(s)$ is continuous and bounded on the finite contour $\Gamma$, then}
	\begin{align}
		\label{eq:MLestimate}
		\left |\int_\Gamma f(s) ds \right | \leq \textrm{length}(\Gamma) \sup_{s\in\Gamma} |f(s)|.
	\end{align}
	\caleb{Therefore, using \eqref{eq:MLestimate}, we can bound the contour integral over $\Gamma_M^{(2)}$ with $\textrm{length}(\Gamma_M^{(2)}) = 2M+1$,}
	$$\left | \int_{\Gamma_M^{(2)}} \hat g_k(-s) \hat h_k(s) ds\right | \leq \frac{2M+1}{M^2-\lambda^2} \sum_{\ell\in\Z} \left |\widetilde \phi_\ell\right |\sum_{\ell\in\Z} \left |\phi_\ell\right | \to 0 \quad \mbox{as} \quad M\to \infty.$$
	Next, we show that $\int_{\Gamma_M^{(1)}} \hat g_k(-s) \hat h_k(s)ds \to 0$ as $M\to\infty$:
	$$\Gamma_M^{(1)} = \{z=x+\i(M+\frac{1}{2}):-M\leq x\leq 0\}.$$
	Then for $s\in \Gamma_M^{(1)}$,
	\begin{align*}
		|\hat g_k(-s) \hat h_k(s)|&=|\hat g_k(x+\i(M+1/2)) \hat h_k(-x-\i(M+1/2))|\\
		&= \left |\sum_{\ell\in\Z} \frac{\widetilde \phi_\ell}{x-\lambda+\i(M+\frac{1}{2}\caleb{+\ell})} \right| \left|\sum_{m\in\Z}\frac{\phi_m}{-x-\lambda+\i(-M-\frac{1}{2}\caleb{+\ell})}\right|\\
		&\leq \sum_{\ell\in\Z}\frac{|\widetilde \phi_\ell|}{|M+\frac{1}{2}\caleb{+\ell}|} \sum_{m\in\Z}\frac{|\phi_m|}{|M+\frac{1}{2}+m|}.
	\end{align*}
	Note that $\mbox{length}(\Gamma_M^{(1)}) = M.$
	The inequality \eqref{eq:MLestimate} yields
	\begin{align*}
		\left |\int_{\Gamma_M^{(1)}} \hat g_k(-s) \hat h_k(s)ds\right | \leq M \sum_{\ell\in\Z}\frac{|\widetilde \phi_\ell|}{|M+\frac{1}{2}\caleb{+\ell}|} \sum_{m\in\Z}\frac{|\phi_m|}{|M+\frac{1}{2}+m|}.
	\end{align*}
	\noindent We need to show that the quantity on the right approaches $0$ as $M\to\infty$.
	To accomplish this, we perform a change of variable on the first summation,
	\begin{align*}
		M \sum_{\ell\in\Z}\frac{|\widetilde \phi_\ell|}{|M+\frac{1}{2}\caleb{+\ell}|} = M\sum_{\hat\ell\in\Z}\frac{|\caleb{\widetilde \phi_{\hat \ell-M}}|}{|\hat\ell+\frac{1}{2}|} \leq 2M \sum_{\hat\ell\in\Z} | \caleb{\widetilde \phi_{\hat\ell - M}}|.
	\end{align*}
	Then we can pass the limit,
	\begin{align*}
		\lim_{M\to\infty} 2M \sum_{\hat\ell\in\Z} | \caleb{\widetilde \phi_{\hat\ell-M}} | = 	2 \sum_{\hat\ell\in\Z} \lim_{M\to\infty} M | \caleb{\widetilde \phi_{\hat\ell-M}} | = 0,
	\end{align*}
	\noindent since
$		M | \caleb{\widetilde \phi_{\hat \ell-M}}| = M |\caleb{\widetilde c_{k-(\hat\ell-M)} \widetilde b_{\hat\ell-M}}| \to 0$
	by Lemma \ref{lemma:fourierpwsmooth}.

	\noindent We can employ a similar argument for $\Gamma_N^{(3)}$ to show that
	\begin{align*}
		\left |\int_{\Gamma_M^{(3)}} \hat g_k(-s) \hat h_k(s)ds\right | \to 0, \qquad \mbox{as } M\to\infty.
	\end{align*}

	\noindent Then we obtain
	\begin{align*}
		\langle \hat g_k,\hat h_k \rangle_{\H2} &= \int_{-\infty}^\infty \hat g_k(-\i\omega) \hat h_k(\i\omega)d\omega 
	= \lim_{M\to\infty}\int_{\gamma_M} \hat g_k(-s) \hat h_k(s)ds \\	
		&= \sum_{\ell\in\Z} \hat g(-\lambda^{(\ell)}){\normalfont\textsf{res}}[\hat h_k(s),\lambda^{(\ell)}].
	\end{align*}
	To consider the case where $\Q\in\R^{n\times n}$ with $n\geq 2$, we need to guarantee that the poles of $\hat g_k$ and $\hat h_k$, i.e., 
		\caleb{$\lambda_j^{(\ell)} = \lambda_j(\Q) - \i \ell \omega_0$ for $\ell\in\Z,$}
	\noindent will not coincide and consequently remain simple.
	The assumption that
			$\max_j |\textrm{Im}\{\lambda_j\}| < \omega_0$
	\noindent is a sufficient condition to keep the poles simple and hence the case for finite dimension $n\geq 2$ is a natural generalization,
	\begin{align*}
		\langle \hat g_k,\hat h_k \rangle_{\H2}
		= \sum_{j=1}^n \sum_{\ell\in\Z} \hat g(-\lambda_j^{(\ell)}){\normalfont\textsf{res}}[\hat h_k(s),\lambda_j^{(\ell)}].
	\end{align*}
The expression for the $\H2$-norm follows immediately. 
\end{proof}

Now we are in a position to give a result analogous to \eqref{eq:h2poleressiso}, but for  \textsf{LTP} systems, representing the $\H2$ inner product and norm using poles and residues.
\begin{theorem}  \label{thm:H2innerLTPpoleres}
Let $\ltp$ and $\ltph$ be two \textsf{LTP} systems as in Theorem \ref{thm:h2poleresinnerprod} sharing the same state-matrix $\bQ$ with $\widetilde{\b}(t)$, $\widetilde{\c}^T(t)$, $\b(t)$, $\c^T(t)$  having components that are 
		 $T$-periodic and in $\mathsf{Lip}(\alpha)$ with $\alpha>\frac12$.
Let $\hat g_k,\hat h_k\in\H2$ be the $k$th subsystems of $\ltp$ and $\ltph$ respectively.
Assume 
  $|Im\{\lambda_j(\Q)\}| < \omega_0$ where  $\omega_0$ denotes the fundamental frequency of $\ltp$ and $\ltph$. Then,
\begin{equation}  \label{eq:prh2inner}
\langle \ltp, \ltph \rangle_{\H2} = 
 \sum_{ k \in \Z} \sum_{\ell\in\Z} \sum_{j=1}^n \hat g_k(-\lambda_j^{(\ell)}){\normalfont\textsf{res}}[\hat h_k(s),\lambda_j^{(\ell)}]
\end{equation}
 and
 	\begin{equation} \label{eq:prh2norm}
		\left\|\ltp\right\|_\H2^2 =  \sum_{k \in \Z} \sum_{\ell\in\Z}\sum_{j=1}^n \hat g_k(-\lambda_j^{(\ell)}){\normalfont\textsf{res}}[ \hat g_k(s),\lambda_j^{(\ell)}],
	\end{equation}
 where $\lambda_j^{(\ell)} = -\lambda_j + \ell\i$ is as defined in Theorem \ref{thm:h2poleresinnerprod}.
\end{theorem}		
\begin{proof}
The result follows from combining Theorem \ref{thm:h2poleresinnerprod} and Theorem \ref{thm:h2innerproduct}. More specifically,  combining \eqref{eq:h2poleresinnerprod} with \eqref{eq:sandberhh2inner} yields \eqref{eq:prh2inner},  and  combining  
\eqref{eq:LTIH2sub_poleres} with \eqref{eq:LTIH2sub}
yields
 \eqref{eq:prh2norm}.
\end{proof}

Theorem \ref{thm:H2innerLTPpoleres} extends  the formulae \eqref{eq:h2poleressiso}-\eqref{eq:h2normpoleressiso} to the \textsf{LTP} setting. \serkan{However, note that the inner product formula in the \textsf{LTP} setting assumes that the systems share the same $\bQ$ matrix.} 
Even though formulae are more complicated including two infinite sums, it is analogous to the \textsf{LTI} case in the sense it is a weighted sum of the residues where the weight is the  subsystem transfer function evaluated at the mirror images of the poles. In the \textsf{LTI} case, this formulation  has lead to the interpolatory $\H2$ optimality conditions for model reduction \cite{Gugercin:2008bc}. This 
 is a potential research direction to pursue for the \textsf{LTP}  setting. 

\section{An $\H2$-based Model Reduction Framework for \textsf{LTP} systems}

Given the full-order  \textsf{LTP} system $\ltp$ as in \eqref{eq:FOM} with state-space dimension $n$, 
we seek a low-order  \textsf{LTP} approximant $\widetilde \ltp$ as in  \eqref{eq:ROM} with state-space dimension $r$, where
 $r\ll n$. The mumerical implementation of the proposed method presented below assumes the availability of a Floquet-transformed equivalent system. Since we know such a transformation exists for every \textsf{LTP} system, without loss of generality, we assume $\A(t) = \Q$ in this section. We note that in some prominent applications, e.g., in the analysis of mechanical systems with moving loads as considered in  \cite{Stykel_Vasilyev}, the full-order 
 system has already a constant  state-mapping matrix $\A(t) = \Q$. In other applications, e.g., nonlinear circuit modeling and analysis, one can compute the Floquet-transformation from the sampled simulation data as we did in our numerical example in Section \ref{sec:noda}.

Our approach uses a Petrov-Galerkin projection framework to perform reduction. Given the \textsf{LTP} system $\ltp$ as in \eqref{eq:FOM} with now $\A(t) = \Q$, we will construct two matrices $\bV,\bW \in \mathbb{R}^{n \times r}$ with $ \W^T \V = \mathbf{I}_r$ such that the reduced \textsf{LTP} system  in  \eqref{eq:ROM} is given by
	$\widetilde \A(t) = \W^T \Q \V$,
	$\widetilde \b(t) = \W^T \b(t)$ and
	$\widetilde \c(t) = \c(t) \V$.
 Our proposed approach converts the \textsf{LTP} model reduction problem into an equivalent \textsf{MIMO} \textsf{LTI} problem, utilizing the optimal, numerically efficient $\mathsf{H}_2$ model reduction techniques for \textsf{LTI} systems to construct projection subspaces, $\bV$ and $\bW$  for the \textsf{LTP} problem.  We are then able to provide an error bound for the approximation error in the original \textsf{LTP} setting.

Stykel and Vasilyev \cite{Stykel_Vasilyev} introduced a model reduction scheme for linear time-varying (LTV) systems that was developed independently of the earlier thesis \cite{magruder2013model} upon which the present work is based.  
Stykel and Vasilyev considered a special class of LTV systems describing mechanical systems with moving loads; 
creating a time-dependent linear dynamical system equivalent in structure to a $\Q$-$\b(t)$-$\c(t)$ realization considered in Theorem 
\ref{thm:h2poleresinnerprod} but for arbitrarily time-varying $\b(t)$ and $\c(t)$.   The approach we pursue here originates with the presumed
periodicity of $\b(t)$ and $\c(t)$, leading to a substantially different approach from that of \cite{Stykel_Vasilyev}.  We are able to take advantage of this presumed periodicity both analytically and computationally.   Indeed, the extended-$\H2$ space of \textsf{LTP} systems
%
%
 introduced above allows for familiar performance guarantees entirely analogous to those in the \textsf{LTI} setting.

\subsection{ Connection of \textsf{LTP} Systems to \textsf{LTI} \textsf{MIMO} Systems}

If $\b(t)$ and $\c(t)$ have finite order Fourier series expansions,
		\begin{align} \label{eq:ff}
			\b(t) = \sum_{k = -N}^N \b_k e^{-\i k\omega_0 t} 	
				~~~~\mbox{and}~~~
			\c(t) = \sum_{k = -N}^N \c_k e^{-\i k\omega_0 t},
		\end{align}
then the  \textsf{LTP} system has  finite number of nontrivial subsystems, each of which have only finite spectra.
This system has $2N+1$ inputs formed by the Fourier coefficients of $\b(t)$ and $2N+1$ outputs formed by the Fourier coefficients of $\c(t)$ resulting in $4N+1$ nontrivial subsystems. The next result connects the $\H2$ norm of such an \textsf{LTP}
system to the $\H2$ norm of  an \textsf{LTI} system.
	\begin{theorem}
		\label{thm:ltpmimobound}
		Let $\displaystyle \ltp = \sys{\Q}{\b(t)}{\c^T(t)}{\mathbf 0}$ be an \textsf{LTP} system where $\b(t)$ and $\c(t)$ have the finite Fourier expansions 	as in \eqref{eq:ff}.
		We define an associated \textsf{LTI} \textsf{MIMO} system 
		\begin{align}
			\label{eq:ltp2mimo}
			\ltih = \sys{\Q}{[\b_{-N},\ldots,\b_N]}{\left [ \c_{-N}, \ldots, \c_N \right ]^T}{\mathbf 0}.
		\end{align}
		Then the $\H2$ norm of the \textsf{LTP} dynamical system $\ltp$ can be bounded in terms of the $\H2$ norm of the \textsf{LTI} dynamical system $\ltih$, namely
		\begin{align*}
			\left\|\ltp\right\|_{\H2} \leq \sqrt{2N+1}\ \left\|\ltih\right\|_{\H2}.
		\end{align*}
	\end{theorem}

	\begin{proof}
		The above system, $\ltp$, has subsystems $\hat g_k(s)$  that can be written
		\begin{align*}
			\hat g_k(s) = \sum_{\ell=-N+k}^{N+k} \c^T_{k-\ell}[s_\ell \I - \Q]^{-1}\b_\ell~~\mbox{for}~~k=-N,\ldots,N.
		\end{align*}		
		Each of these terms are finite-dimensional \textsf{LTI} systems.
		Using equation \eqref{eq:LTIH2sub},
		\begin{align*}
			\left\|\ltp\right\|_{\H2}^2 &= \sum_{k=-2N}^{2N} \left\|\hat g_k\right\|_{\H2}^2
			=\sum_{k=-2N}^{2N} \left \| \sum_{\ell=-N+k}^{N+k} \c^T_{k-\ell}[s_\ell \I - \Q]^{-1}\b_\ell \right \|_{\H2}^2.
		\end{align*}
		Now we use that $\|a_1 + \ldots + a_K\|^2 \leq \serkan{K} (\|a_1\|^2 + \ldots \|a_K\|^2)$ (application of the triangle inequality and the property, $2ab \leq a^2 + b^2$ for $a,b\in\R$) to conclude
		\begin{align*}
			\left \|\sum_{\ell=-N+k}^{N+k} \c^T_{k-\ell}[s_\ell \I - \Q]^{-1}\b_\ell \right \|^2_{\H2} \leq (2N+1) \sum_{\ell=-N+k}^{N+k} \| \c^T_{k-\ell}[s_\ell \I - \Q]^{-1}\b_\ell\|^2_{\H2}.
		\end{align*} 
		Then we have
		\begin{align*}
			\left\|\ltp\right\|^2_{\H2} &\leq (2N+1) \sum_{k=-2N}^{2N} \sum_{\ell=-N+k}^{N+k} \| \c^T_{k-\ell}[s_\ell \I - \Q]^{-1}\b_\ell\|^2_{\H2}\\
			&= (2N+1) \sum_{i,j=-N}^{N} \| \c^T_{i}[s_j \I - \Q]^{-1}\b_j\|^2_{\H2},
		\end{align*}
		\noindent where we recognize the last term as the $\H2$ norm of the \textsf{MIMO} system,
		\begin{align*}
			\ltih = \sys{\Q}{[\b_{-N},\ldots,\b_N]}{ [\c_{-N},\ldots,\c_N ]^T }{\mathbf 0}.
		\end{align*}
		Thus, $\left\|\ltp\right\|^2_{\H2} \leq (2N+1)\left\|\ltih\right\|_{\H2}^2$.
	\end{proof}

	This result leads to the model order reduction algorithm for \textsf{LTP} systems outlined in the next section.

\subsection{A Model Reduction Method for a Special Case of \textsf{LTP} Systems}
	Inspired by Theorem \ref{thm:ltpmimobound}, here we introduce a model-reduction scheme that works for a special case of \textsf{LTP} systems with time-invariant $\A(t) = \Q$.
	We begin by introducing a performance guarantee of reduced-order \textsf{LTP} systems constructed from projection matrices $\V$ and $\W$.
	By Theorem \ref{thm:ltpmimobound}, we immediately obtain the following corollary.
	\begin{corollary}
	\label{thm:MIMObound}
		Given \caleb{an} \textsf{LTP} system of the form $\ltp = \sys{\Q}{\b(t)}{\c^T(t)}{\mathbf 0}$ where $\b(t)$ and $\c^T(t)$ have finite Fourier expansions as in \eqref{eq:ff}. Let $\V,\W\in\R^{n \times r}$ be projection matrices such that $\W^T \V = \I$ and define the reduced \textsf{LTP} system
		\begin{align*}
			\widetilde{\ltp} = \sys{\W^T\Q\V}{\W^T\b(t)}{\c^T(t)\V}{\mathbf 0} = \sys{\widetilde{\Q}}{\widetilde{\b}(t)}{\widetilde{\c}^T(t)}{\mathbf 0}.
		\end{align*}
		Then
		\begin{align}
			\label{eq:error_bd}
			\left\|\ltp-\widetilde \ltp\right\|_{\H2} \leq \sqrt{2N+1} \left\|\ltih-\widetilde \ltih\right\|_{\H2},
		\end{align}
		where $\ltih$ and $\widetilde \ltih$ are the \textsf{LTI} \textsf{MIMO} counterparts, i.e.,
		\begin{align*}
			\ltih &= \sys{\Q}{[\b_{-N},\ldots,\b_N]}{ [\c_{-N}, \ldots, \c_N]^T }{\mathbf 0}, ~~~\mbox{and}\\
		\widetilde	\ltih &= \sys{\W^T\Q\V}{\W^T[\b_{-N},\ldots,\b_N]}{ [\c_{-N}, \ldots, \c_N]^T\V }{\mathbf 0}.
		\end{align*}
	\end{corollary}
	Corollary \ref{thm:MIMObound} connects the model reduction error resulting from the analogous \textsf{LTI~MIMO} problem to the original \textsf{LTP} model reduction problem. This gives a lot of flexibility since the tools for   \textsf{LTI} model reduction are  well established and one can 	employ various model reduction methods, such as  
Balanced Truncation \cite{MulR76,Moo81}, optimal Hankel Norm Approximation \cite{glover1984all}, or Iterative Rational Krylov Algorithm (\textsf{IRKA}) \cite{Gugercin:2008bc}. However, since the error bound is given in terms of the $\H2$ error and \textsf{IRKA} produces locally optimal approximations in the $\H2$ sense, \textsf{IRKA} is a natural candidate for producing $\widetilde \ltih$ that minimizes $\|\ltih - \widetilde \ltih\|_{\H2}$. 	
	
	Corollary \ref{thm:MIMObound} analyzed the case for finite Fourier expansion.  The error analysis for the case with infinitely many Fourier coefficients is given next. 
	\begin{corollary}
	Given an \textsf{LTP} system of the form $\ltp = \sys{\Q}{\b(t)}{\c^T(t)}{\mathbf 0}$ with Fourier expansions $\b(t) = \sum_{-\infty}^\infty \b_k e^{\i k \omega_0 t}$ and $\c(t) = \sum_{-\infty}^\infty \c_k e^{\i k \omega_0 t}$, \serkan{let 
$\ltp_{[N]}$ denote the Fourier truncated \textsf{LTP} system
		\begin{align*}
			\ltp_{[N]} = \sys{\Q}{\sum_{k = -N}^N \b_k e^{\i k\omega_0 t}}{ \sum_{k = -N}^N \c^T_k e^{\i k\omega_0 t}}{\mathbf 0},
		\end{align*}
		and
		the system $\ltih_{[N]}$ is the associated \textsf{LTI} \textsf{MIMO} system for $\ltp_{[N]}$,
		\begin{align*}
			\ltih_{[N]} = \sys{\Q}{[\b_{-N},\ldots,\b_N ]}{ [\c_{-N}, \ldots, \c_N]^T }{\mathbf 0}.
		\end{align*}}
		
\noindent Moreover, let $\V,\W\in\R^{n \times r}$ be projection matrices such that $\W^T \V = \I$ and define the reduced system
	 \serkan{
		\begin{align*}
			\widetilde{\ltp} =  \sys{\W^T\Q\V}{\sum_{k = -N}^N \W^T\b_k e^{\i k\omega_0 t}}{ \sum_{k = -N}^N \c^T_k \V e^{\i k\omega_0 t}}{\mathbf 0}.
		\end{align*}}
		
		\noindent Then the dynamical system \caleb{error $\|\ltp-\widetilde \ltp_{[N]}\|_{\H2}$ is} bounded by
		\begin{align}
			\left\|\ltp-\widetilde  \ltp_{[N]}\right\|_{\H2} &\leq\left \|\ltp-\ltp_{[N]}\right\|_{\H2} + \left\|\ltp_{[N]} - \widetilde  \ltp_{[N]}\right\|_{\H2} \nonumber \\
			&\leq \left\|\ltp-\ltp_{[N]}\right\|_{\H2} + \sqrt{2N+1}\left\|\ltih_{[N]} - \widetilde{\ltih}_{[N]}\right\|_{\H2}, \label{eq:ub}
		\end{align}
		where
		 $\widetilde \ltih_{[N]}$ is the Petrov-Galerkin approximation of $\ltih_{[N]}$, i.e.,
		\begin{align*}
			\widetilde \ltih_{[N]} &= \sys{\W^T\Q\V}{\W^T [\b_{-N},\ldots,\b_N ]}{ [\c_{-N}, \ldots, \c_N]^T \V}{\mathbf 0}.
		\end{align*}
	\end{corollary}

	
	A sketch of the proposed model reduction algorithm is given in Algorithm \ref{PropMeth}.
\begin{algorithm}
	\caption{An $\H2$-based \textsf{LTP} Model Reduction Algorithm} \label{PropMeth}
		Given an  \textsf{LTP} system of the form $\ltp = \sys{\Q}{\b(t)}{\c^T(t)}{\mathbf 0}$:
		\begin{enumerate}
			\item Truncate Fourier series $\displaystyle \b(t) \approx \sum_{k = -N}^N \b_k e^{\i k\omega_0 t}$ and $\displaystyle \c(t) \approx \sum_{k = -N}^N \c_k e^{\i k\omega_0 t}$. Define
			\begin{align*}
				\ltp_{[N]} = \sys{\Q}{\sum_{k = -N}^N \b_k e^{\i k\omega_0 t}}{ \sum_{k = -N}^N \c^T_k e^{\i k\omega_0 t}}{\mathbf 0}
			\end{align*}
			\item Construct the associated \textsf{MIMO} system,
			\begin{align*}
				\ltih_{[N]} = \sys{\Q}{[\b_{-N},\ldots,\b_N ]}{ [\c_{-N}, \ldots, \c_N]^T }{\mathbf 0}
			\end{align*}
			from the Fourier coefficients of $\ltp_{[N]}$.
			\item Use \textsf{IRKA} to find projection matrices, $\V$ with $\W$, $\W^T \V = \I$, that locally minimizes $\|\ltih_{[N]} - \widetilde \ltih_{[N]}\|_{\H2}$ for $r$-dimensional
			\begin{align*}
				\widetilde \ltih_{[N]} &= \sys{\W^T\Q\V}{\W^T [\b_{-N},\ldots,\b_N ]}{ [\c_{-N}, \ldots, \c_N]^T \V}{\mathbf 0}\\
				&= \sys{\widetilde \Q}{[\widetilde\b_{-N},\ldots,\widetilde\b_N ]}{ [\widetilde\c_{-N}, \ldots, \widetilde\c_N]^T }{\mathbf 0}.
			\end{align*}
			\item Construct approximating \textsf{LTP} system $\widetilde  \ltp_{[N]}$ from the coefficients of $\widetilde \ltih_{[N]}$
			\begin{align*}
				\widetilde  \ltp_{[N]} = \sys{\widetilde \Q}{\displaystyle \sum_{k = -N}^N \widetilde\b_k e^{\i k\omega_0 t}}{\displaystyle \sum_{k = -N}^N \widetilde\c^T_k e^{\i k\omega_0 t}}{\mathbf 0}
			\end{align*}
		\end{enumerate}
\end{algorithm}

\section{Numerical Results}

	We demonstrate the  proposed model reduction scheme on \serkan{three} examples: \caleb{(1) a heat equation with a moving point source,
	(2) a nonlinear transmission line \textsf{LTP} model, and 
	 (3) a constructed example, the structural model of component 1r (Russian service module) of the International Space Station.}
As most of the linear-time varying model reduction methods are  computationally challenging for the problems we would like to consider, we compare our method to 
the Linear Time Varying Balanced Truncation method of \cite{lang2015towards} and
 Proper Orthogonal Decomposition (\textsf{POD}) \cite{kunisch2001galerkin}, which remains computationally  tractable even for large-scale problems since it only requires a time-domain simulation.   
 \serkanlast{Steih and Urban \cite{steih2012space} introduce a space-time reduced basis method for time-periodic parametric partial differential equations. However, since we only consider nonparametric \textsf{LTP} systems here, for our purposes, it is enough to consider regular \textsf{POD}.
}
 
\serkan{ The proposed method as described in Algorithm \ref{PropMeth} applies \textsf{IRKA} to a system with $2N+1$ inputs and outputs. As mentioned briefly in Section \ref{sec:irkah2},  the convergence  of \textsf{IRKA} may slow down as the number of inputs and outputs increases. Thus, for modest $N$, one might expect that  \textsf{IRKA} in Step 3 of Algorithm \ref{PropMeth} may be slow to converge.
The residue correction step introduced by Beattie and Gugercin in \cite{beattie2012realization} has largely resolved this issue and improved the MIMO behavior of IRKA significantly; see \cite{beattie2012realization}.   For the problems studied here, we have found that the (regular) \textsf{IRKA} algorithm converged after a modest number of iterations even without the residue correction methodology of \cite{beattie2012realization}.}

	\subsection{Computation of the $\H2$  norm for error comparisons}  \label{sec:h2compute}
Even though Algorithm \ref{PropMeth} does not require computing the $\H2$ norm at any point, in order to provide a detailed comparison to the reader between the reduced models resulting from Algorithm \ref{PropMeth} and \textsf{POD}, we present the resulting $\H2$ error norms. However, computing the $\H2$ error of \textsf{LTP} systems is a nontrivial exercise and we discuss a practical implementation based on Zhou and Hagiwara, \cite{Zhou:2002tp}.
		
	To compute the $\H2$ norm of \caleb{an} \textsf{LTP} system $\ltp = \sys{\Q}{\b(t)}{\c^T(t)}{}$, construct a new system, $\mathfrak G_N = \sys{\mathfrak A_N}{\mathfrak B_N}{\mathfrak C_N}{}$ from a finite number of Floquet-Fourier coefficients $\{\b_i\}_{-N}^N$ and $\{\c_i\}_{-N}^N$ where

	{\small
		\begin{equation}
		\label{eq:ZhouHagiwaraMIMO}
		\mathfrak C_N = \begin{bmatrix}
 			\c_{-N}^T & & & & 0\\
 			\vdots & \ddots\\
 			\c_0^T & & \c_{-N}^T\\
 			\vdots & \ddots & \vdots & \ddots\\
 			\c_N^T & \cdots & \c_0^T & \cdots & \c_{-N}^T\\
 			& \ddots & \vdots & \ddots & \vdots\\
 			& & \c_N^T & & \c_0^T\\
 			& & & \ddots & \vdots\\
 			0 & & & & \c_N^T 
 		\end{bmatrix},
				\begin{array}{l}
		\mathfrak A_N =  \mbox{blkdiag}(\caleb{\Q-\i N \omega_0 \I,\ldots,\Q+\i N \omega_0 \I}),\\ \\
		\mathfrak B_N =  [\b^T_{-N},\ldots,\b^T_{-1},\b^T_0,\b^T_1,\ldots,\b^T_N]^T.
		\end{array}
	\end{equation}	
    	}
    	
	\noindent Then we can approximate the \textsf{LTP} $\H2$ norm $\|\ltp\|_{\H2}$ with arbitrary accuracy as we keep more of the Floquet-Fourier coefficients.

	\begin{theorem}[Zhou \& Hagiwara, \cite{Zhou:2002tp}]
		\label{thm:h2norm_comp}
		If $\mathfrak G_N = \sys{\mathfrak A_N}{\mathfrak B_N}{\mathfrak C_N}{}$ where $\mathfrak A_N$, $\mathfrak B_N$ and $\mathfrak C_N$ are defined according to \eqref{eq:ZhouHagiwaraMIMO}, then
		\begin{align*}
			\lim_{N\to\infty} \mbox{trace}(\mathfrak B_N^* \mathcal V_N \mathfrak B_N) = \lim_{N\to\infty} \mbox{trace}(\mathfrak C_N \mathcal W_N \mathfrak C_N^*) = \|\ltp\|_{\H2}^2
		\end{align*}
		where $\mathcal V_N$ and $\mathcal W_N$ are, respectively, the solutions of the finite-dimensional Lyapunov equations
		\begin{align}
			\label{eq:Lyap}
			\mathfrak A_N^* \mathcal V_N + \mathcal V_N \mathfrak A_N + \mathfrak C_N \mathfrak C_N^* = 0 ~~\mbox{and}~~
			\mathfrak A_N \mathcal W_N + \mathcal W_N \mathfrak A_N^* + \mathfrak B_N \mathfrak B_N^* = 0.
		\end{align}
	\end{theorem}
	
	We note that in practice the Lyapunov equations \eqref{eq:Lyap} can be enormous in size as $\mathfrak A_N \in \R^{Nn\times Nn}$ where $N$ is the number of Fourier coefficients retained and $n$ is the order of the system.
	However, due to the block diagonal structure of $\mathfrak A_N$, the dominant cost of the Lyapunov solver is the Schur decomposition of $\Q$, which can be computed once and reused for \caleb{$\Q - \i k \omega_0\I$}. 
	Then the back substitution in the Bartels-Stewart algorithm can be done on each block individually, avoiding the cost of an $Nn \times N n$ Schur decomposition.
	For description of the Bartels-Stewart algorithm, see Sorensen and Zhou \cite{sorensen2003direct}.
	We found this to be a critical step in making the $\H2$ norm computations feasible for illustrating the model reduction errors.
	
\serkanlast{In the three numerical examples below, to illustrate the formula \eqref{eq:ub}, we plot the $\mathcal{H}_2$ error and its upper bound. However, as the discussion above illustrates, computing the exact Fourier truncation error $\|\ltp-\ltp_{[N]}\|_{\H2}$ is not numerically feasible if $\ltp$ has an infinite Fourier expansion. Indeed, in the numerical example of Section \ref{sec:noda}, we do not even have full access to $\ltp$. Therefore, in Sections \ref{ex:heat} and  \ref{sec:noda}, we keep all the $N$ Fourier coefficient that we are able to obtain via a numerical simulation and treat this system approximately as $\ltp$, i.e., $\ltp \approx \ltp_{[N]}$.	This means that in Sections  \ref{ex:heat} and  \ref{sec:noda}, the $\mathcal{H}_2$  norm plots show the   $\mathcal{H}_2$ error  $\left\|\ltp_{[N]} - \widetilde  \ltp_{[N]}\right\|_{\H2}$ together with its upper bound 
			$ \sqrt{2N+1}\left\|\ltih_{[N]} - \widetilde{\ltih}_{[N]}\right\|_{\H2}$. On the other, the  numerical example of Section \ref{ex:iss} is constructed such that the \textsf{LTP} system  has only $N=5$ Fourier coefficients. Therefore, in that example we have $\ltp = \ltp_{[N]}$ and the norm computations are exact.}	

\subsection{\caleb{1D  Heat Model}} \label{ex:heat}

\caleb{
The following model is a modified version of Example 1 in \cite{lang2015towards} of the 1D heat equation,
\begin{align*}
	\frac{\partial z}{\partial t}(t,x) - \frac{\partial^2 z}{\partial x^2}(t,x) &= \delta(x- \xi(t))u(t), \qquad (t,x) \in (0,T) \times (0,1),\\
	z(t,0) = z(t,1) &= 0, \qquad\qquad\qquad\qquad t \in (0,T),\\
	z(0,x) &= 0, \qquad\qquad\qquad\qquad x \in (0,1),\\
	y(t) &= z(t,0.5), \quad\qquad\qquad t \in (0,T),
\end{align*}
\noindent with a moving point source where $\delta(t)$ is the Dirac delta function and $\xi(t)$ and $u(t)$ denote the heat source position and thermal flux respectively.
}
\caleb{
The PDE is discretized via a finite difference discretization with $2502$ equidistant grid points.
Due to the Dirichlet boundary condition the discretized system has $n=2500$ degrees of freedom.
Simulations are discretized in time via Backwards Euler with final time $T=100$ [sec] and $\Delta t = 1$ [sec].
To make the source term periodic, we choose $\xi(t) = 0.5 + 0.4 \sin(8\pi t / T)$.
}

\caleb{For this model we  include a comparison to \textsf{POD} and the Linear Time-Varying Balanced Truncation (\textsf{LTV BT}) method introduced by Lang et al.\ \cite{lang2015towards} 
This approach solves the two Lyapunov equations at every time step and produces different ROM trial and test spaces ($\V(t_i)$ and $\W(t_i)$) for every time step as well.
While \textsf{LTV BT} produces very accurate reduced models, solving many Lyapunov equations is  expensive.
The authors alleviate this issue with an iterative process that warm-starts the Lyapunov solver at each time step with the solution from the previous time step; however, the computational cost of \textsf{LTV BT} is still considerably larger than our approach as is demonstrated in the numerical results that follow.
It is important to note that our approach exploits periodicity of the state space representation, whereas \textsf{LTV BT} does not and indeed can handle a broader class of time-varying problems than we  consider here.}

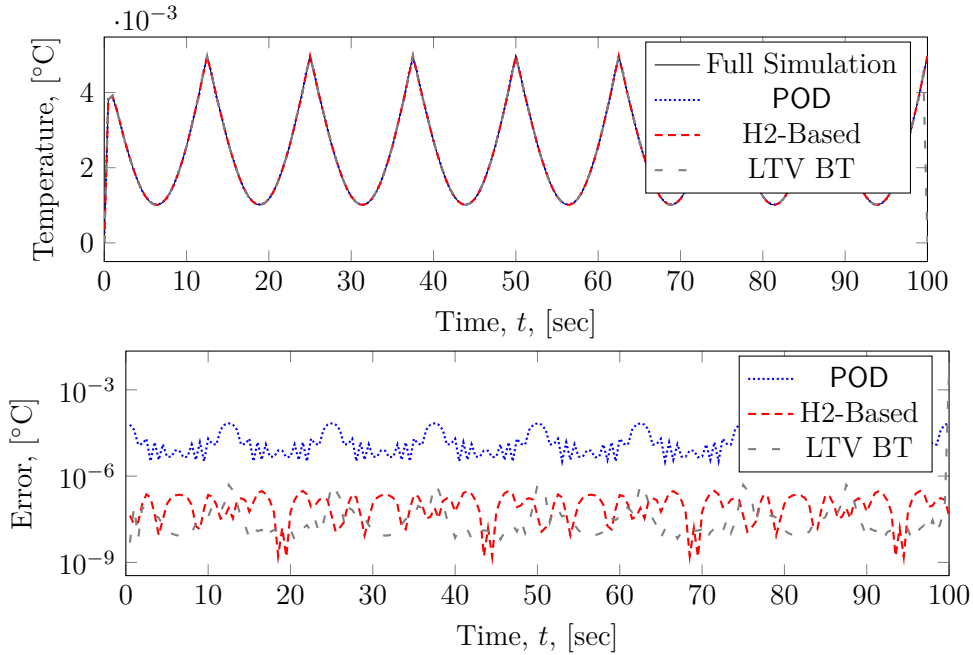
\begin{figure}[H]
	\centering
	\begin{tikzpicture}
		\begin{axis}[
			width = 0.9\textwidth,
			height = 0.2\textheight,
			xlabel = {Time, $t$, [sec]},
			ylabel = {Temperature, [$^\circ$C]},
			xmin = 0,
			xmax = 100,
			legend entries = {Full Simulation, \textsf{POD}, H2-Based, LTV BT},
			]
			\addplot+[mark = none,black] table [x index=0, y index=1, col sep=comma] {timesim.dat};
			\addplot+[mark = none,thick,blue,densely dotted] table [x index=0, y index=2, col sep=comma] {timesim.dat};
			\addplot+[mark = none,thick,red,densely dashed] table [x index=0, y index=3, col sep=comma] {timesim.dat};
			\addplot+[mark = none,thick,gray,loosely dashed] table [x index=0, y index=4, col sep=comma] {timesim.dat};
		\end{axis}
	\end{tikzpicture}
	\begin{tikzpicture}
		\begin{semilogyaxis}[
			width = 0.9\textwidth,
			height = 0.2\textheight,
			xlabel = {Time, $t$, [sec]},
			ylabel = {Error, [$^\circ$C]},
			xmin = 0,
			xmax = 100,
			legend entries = {\textsf{POD},H2-Based, LTV BT},
			]
			\addplot+[mark = none,densely dotted,thick] table [x index=0, y index=5, col sep=comma] {timesim.dat};
			\addplot+[mark = none,densely dashed,thick] table [x index=0, y index=6, col sep=comma] {timesim.dat};
			\addplot+[mark = none,gray,loosely dashed,thick] table [x index=0, y index=7, col sep=comma] {timesim.dat};
		\end{semilogyaxis}
	\end{tikzpicture}
	\caption{Transient simulation and error with input of the constant function $u(t) \equiv 1$ for $t \geq 0$. Full order model is of dimension $n=2500$.}
	\label{fig:heattimesim}
\end{figure}

\serkanlast{We train \textsf{POD} with the constant input function $u(t)=1$. Both \textsf{POD} and H2-Based reduced models use $r=14$ for dimension of the reduced system. On the other hand, since the model reduction bases vary at every time step,  \textsf{LTV BT} uses a varying reduced order throughout the simulation, $5 \leq r \leq 9$, depending on the time step. The \textsf{LTV BT} simulations take considerably longer to run ($1355$ [sec]) whereas \textsf{POD} and the proposed $\mathcal{H}_2$-Based approach, i.e.,  Algorithm \ref{PropMeth}, each took less than five seconds.
In Figure~\ref{fig:heattimesim}, we plot the output $y(t)$ (the top plot) and the output errors (the bottom plot) due to the three reduced models obtained via \textsf{POD}, \textsf{LTV BT}, and the proposed method (labeled as ``H2-Based").  The input $u(t)$ for these simulations is the same input function that was used to train \textsf{POD}. 
The first observation is that the input/output based approaches outperform \textsf{POD} as illustrated by the error plot in Figure~\ref{fig:heattimesim}. For this example, \textsf{LTV BT} and the H2-based proposed approach perform similarly. However,  one might expect  \textsf{LTV BT} to outperform the proposed approach in general since it uses time varying model reduction bases $\V(t_i)$ and $\W(t_i)$ at every time step; as opposed to the proposed approach where the model reduction bases are fixed. Therefore, it is encouraging that the proposed method is able to mimic the accuracy of  \textsf{LTV BT} for this example. 
}\caleb{Since \textsf{LTV BT} has time varying bases $\V(t_i)$ and $\W(t_i)$, the resulting reduced model does not allow $\mathcal{H}_2$ norm computations. Therefore, we compute the $\mathcal{H}_2$ error norms only for  \textsf{POD} and the H2-based method. Results in Figure~\ref{fig:heat_r_v_error} show the $\H2$ error as the dimension of the ROM increases.
To make the $\H2$ computation affordable, we have chosen a FOM with $n=100$ degrees of freedom. Once again the proposed method significantly outperforms \textsf{POD}. }
\begin{figure}[ht!]  
	\centering
	\begin{tikzpicture}
		\begin{semilogyaxis}[
			width = 0.9\textwidth,
			height = 0.3\textheight,
			xlabel = {Dimension of Reduced System, $r$},
			ylabel = {$\H2$ Error},
			xmin = 4,
			xmax = 24,
			legend entries = {H2-Based Error,\textsf{POD} Error,H2-Based Error Bd,\textsf{POD} Error Bd,},
			]
			\addplot+[mark = square,thick,blue] table [x index=0, y index=1, col sep=comma] {error_vs_r.dat};
			\addplot+[mark = x,thick,red] table [x index=0, y index=2, col sep=comma] {error_vs_r.dat};
			\addplot+[mark = square,dashed,thick,blue] table [x index=0, y index=3, col sep=comma] {error_vs_r.dat};
			\addplot+[mark = x,dashed,thick,red] table [x index=0, y index=4, col sep=comma] {error_vs_r.dat};
		\end{semilogyaxis}
	\end{tikzpicture}
	\caption{$\H2$ Error and \emph{a posteriori} error bounds of the H2-based method vs \textsf{POD}. Full order model is of dimension $n=100$.}
	\label{fig:heat_r_v_error}
\end{figure}
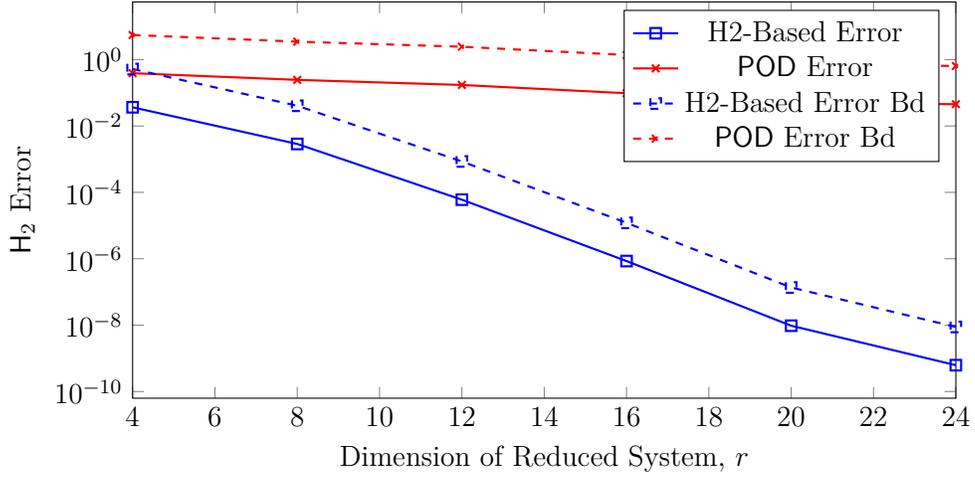

\subsection{Nonlinear Transmission Line} \label{sec:noda}

The following model is the ``Test Network 2" from Noda et al.\ \cite{noda2003harmonic}.
It represents an RLC circuit with nonlinear inductors to model the effect of saturable transformers.
The network is split into 10 sections, each consisting of 6 states.
\caleb{Therefore,} the full model has 60 states representing the current and voltages through each section.

In Figure \ref{fig:nodaschem}, $e = E_m \cos(\omega_0 t)$ where $\omega_0 = 2\pi 60$ \caleb{[rad/sec]} and $E_m = \sqrt{2/3} \times V_{p.u.} \times 500$ \caleb{[kV]}.
The circuit is supplied with an ``overvoltage" by a factor of $V_{p.u.} = 1.25$ to induce a noticeable saturation in the nonlinear inductors, $L_N$.
The current through $L_N$ is $i_N = \alpha\psi + \beta\psi^7$ where $\psi$ is the magnetic flux through the inductor and $\alpha$ and $\beta$ depend on the section.
For a table of the resistor, capacitor, and inductor values, see Figure 8 in Noda et al.\ \cite{noda2003harmonic}.
See Figure \ref{fig:nodaschem} for the nonlinear circuit schematic.

\begin{figure}[ht!]
	\centering
	\includegraphics[scale = 0.9]{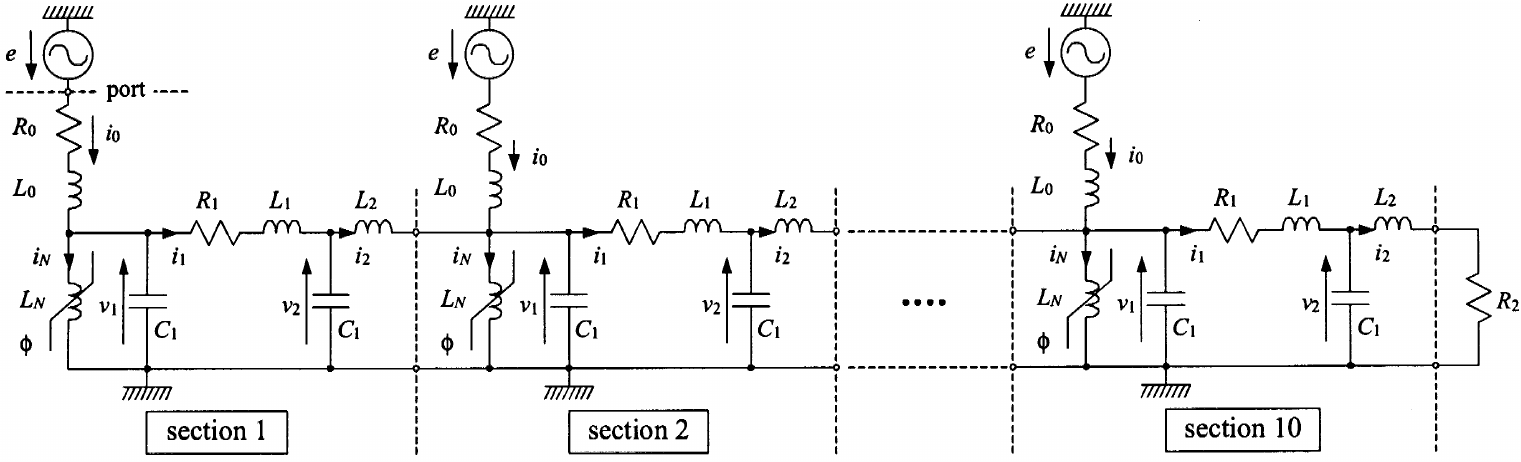}	
	\caption{Schematic for nonlinear circuit taken from Noda et al.\ \cite{noda2003harmonic}. Reproduced with permission.}
	\label{fig:nodaschem}
\end{figure}

The steady-state solution to the problem is \caleb{time-periodic}.
Noda et al.\ \cite{noda2003harmonic} are concerned with the transient deviations from the steady-state solution from perturbations of the input $e\mapsto e + u$.
\caleb{The voltage $e$ and current $i_0$ are the input and output of the dynamical system. As both quantities are degrees of freedom in the dynamical system, the input matrix $\b \in \R^{60}$ and output matrix $\c \in \R^{60}$ do not vary with time.}
\caleb{Therefore,} a linearization of the nonlinear dynamical system results in a time-periodic system of the form \eqref{eq:FOM} where $\A(t),\b,\c$ are $T$-periodic with $T = 1/60$ \caleb{[sec]} and $\x(t)\in \R^{60}$ represents the perturbation in the state from the \serkan{steady}-state solution and $u(t)\in \R$ represents the perturbation from $e$.

\serkanlast{Unlike the previous 1D Heat model, the full \textsf{LTP} models in this example and the next are  stiff, requiring many time-samples to resolve the numerical simulation, even with the implicit backward Euler scheme. This results in significantly more Lyapunov equations to solve increasing the cost of  \textsf{LTV BT} even further compared to the previous example. Therefore, for these last two examples, we only provide comparison to  \textsf{POD}. However, even though \textsf{LTV BT} might be computationally more intensive, we still accept it to provide very accurate reduced models due to the time-varying model reduction bases.}
To run \textsf{POD}, we generated simulations of the system when the control was the Heaviside function, $u(t) = H(10^{-3} - t)$ [kV].
After generating the numerical simulation of the state $\x(t)$, we collect the simulation into a snapshot matrix, $\X = [\bx_{t_1},\ldots,\bx_{t_m}]$ and perform an SVD to construct the subspace $\V \in \R^{n\times r}$.
The result of the reduced-order model generated according to \eqref{eq:ROM} with $\W = \V$ from above is labeled as \textsf{POD} in Figure \ref{fig:nodatimesim}.

\begin{figure}[H]
	\centering
	\begin{tikzpicture}
		\begin{axis}[
			width = 0.9\textwidth,
			height = 0.2\textheight,
			xlabel = {Time, $t$, [sec]},
			ylabel = {Current, [amps]},
			xmin = 0,
			xmax = 0.0166,
			legend entries = {Full Simulation, \textsf{POD}, H2-Based},
			]
			\addplot+[mark = none,black] table [x=time, y=y_og, col sep=comma] {c_timesim.dat};
			\addplot+[mark = none,thick,blue,densely dotted] table [x=time, y=y_pod, col sep=comma] {c_timesim.dat};
			\addplot+[mark = none,thick,red,densely dashed] table [x=time, y=y_krylov, col sep=comma] {c_timesim.dat};
		\end{axis}
	\end{tikzpicture}
	\begin{tikzpicture}
		\begin{semilogyaxis}[
			width = 0.9\textwidth,
			height = 0.2\textheight,
			xlabel = {Time, $t$, [sec]},
			ylabel = {Error, [amps]},
			xmin = 0,
			xmax = 0.0166,
			ymin = 0.001,
			ymax = 10,
			legend entries = {\textsf{POD},H2-Based},
			]
			\addplot+[mark = none,densely dotted,thick] table [x=time, y=y_pod_err, col sep=comma] {c_timesim.dat};
			\addplot+[mark = none,densely dashed,thick] table [x=time, y=y_krylov_err, col sep=comma] {c_timesim.dat};
		\end{semilogyaxis}
	\end{tikzpicture}
	\caption{Transient Simulation of linearized nonlinear circuit, $u(t) = H(10^{-3}-t)$ [kV] where $H$ denotes the Heaviside function.}
	\label{fig:nodatimesim}
\end{figure}
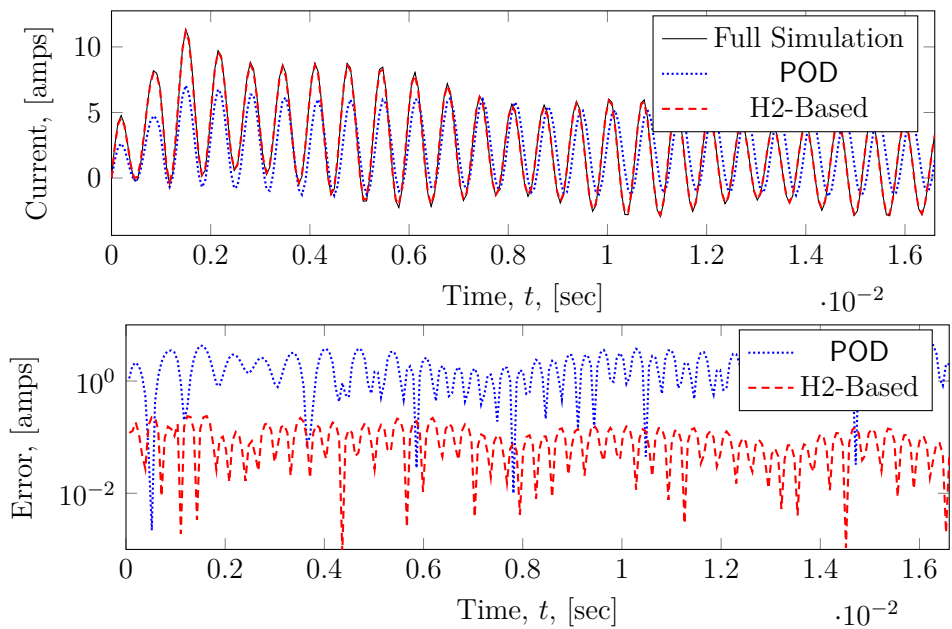

For the model reduction via Algorithm \ref{PropMeth}, we begin by performing a Floquet transformation on the periodic $\bA(t)$.
For this example, we have the \caleb{snapshots $\A(t_i)$ for} 256 uniformly spaced temporal points in the period, $t_i\in [0,1/60]$ \caleb{[sec]}. 
\caleb{Therefore,} the Floquet-Fourier transformation resulted in a state-space matrices of the form: $\Q\in \R^{60 \times 60}$, $\B,\mathbf C\in \R^{60\times 256}$.
\caleb{We preserve} all 256 Fourier modes of the input and output and used \textsf{IRKA} to construct the Petrov-Galerkin subspaces $\W$ and $\V$.

Figure~\ref{fig:nodatimesim} shows the outputs (the top plot) and the output errors (the bottom plot) in time domain simulations
due to two reduced models obtained via \textsf{POD} and the proposed method (labeled as ``H2-Based"), each with reduced order $r=10$. We also perform a similar comparison for a variety of reduced order sizes in Figure \ref{fig:noda_r_v_error}, measuring the $\H2$ error of the \textsf{LTP} systems.
Note that the \textsf{POD}-based reduced model becomes unstable for $r \geq 12$; hence, the $\H2$ error is infinite in these cases and are not plotted.

\begin{figure}[ht!]  
	\centering
	\begin{tikzpicture}
		\begin{semilogyaxis}[
			width = 0.9\textwidth,
			height = 0.3\textheight,
			xlabel = {Dimension of Reduced System, $r$},
			ylabel = {$\H2$ Error},
			xmin = 2,
			xmax = 28,
			legend entries = {\textsf{POD} Error,H2-Based Error,\textsf{POD} Error Bd,H2-Based Error Bd},
			]
			\addplot+[mark = square,thick,blue] table [x=r, y=ltpnormpod, col sep=comma] {c_error_vs_r.dat};
			\addplot+[mark = x,thick,red] table [x=r, y=ltpnormirka, col sep=comma] {c_error_vs_r.dat};
			\addplot+[mark = square,dashed,thick,blue] table [x=r, y=ltpnormpod_bd, col sep=comma] {c_error_vs_r.dat};
			\addplot+[mark = x,dashed,thick,red] table [x=r, y=ltpnormirka_bd, col sep=comma] {c_error_vs_r.dat};
		\end{semilogyaxis}
	\end{tikzpicture}
	\caption{$\H2$ Error and \emph{a posteriori} error \caleb{bound \eqref{eq:error_bd} of the} H2-based method vs \textsf{POD} for the linearization of the nonlinear circuit in from Noda et al.\ \cite{noda2003harmonic} in Figure \ref{fig:nodaschem}. }
	\label{fig:noda_r_v_error}
\end{figure}
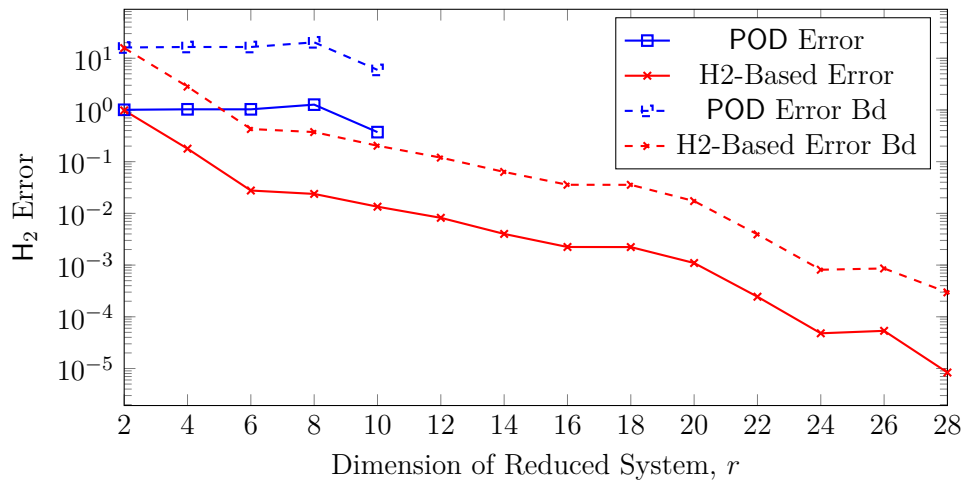
As Figures \ref{fig:nodatimesim} and \ref{fig:noda_r_v_error} illustrate the proposed  approach demonstrates superior performance to \textsf{POD}, in terms of stability and dynamical system error measured in the $\H2$ norm for \textsf{LTP} systems.
In terms of both the time domain simulation and the $\H2$  error norm, the error due to the proposed method is almost two orders of magnitude smaller than that of \textsf{POD}.
We note that this superior performance is obtained in the best case scenario for \textsf{POD} since the model reduction error is measured for the same input which was used to train  \textsf{POD}.
Still the $\H2$-based proposed method produces a significantly better reduced model. 
We note also that the error bound predicts the true error behavior well. 

\subsection{Structural Model of Component 1r (Russian service module) of the International Space Station} \label{ex:iss}
	We consider the 1r Russian service module  that has $270$ states, $3$ inputs and $3$ outputs.
	We make this model \textsf{LTP} by placing modulators on inputs $2$ and $3$. We do the same for outputs $2$ and $3$.
	Define $\mathbf B = \left [\begin{array}{ccc} \b_0 & \b_1 & \b_2 \end{array}\right ]$ to be the original input vector and $\mathbf C = [\c_0, \c_1, \c_2]^T$ to be the original output vector.
	Then we construct a \textsf{SISO} system by feeding the \caleb{input $u(t)$ into} two modulators with local oscillator frequencies of $\omega_0$ and $2\omega_0$ respectively,
	\begin{align*}
		\ltp = \sys{\Q}{\b_0 + \b_1\cos(\omega_0 t) + \b_2\cos(2\omega_0 t)}{\c_0 + \c_1\cos(\omega_0t) + \c_2\cos(2\omega_0t)}{\mathbf 0}.
	\end{align*}

	To see this abstracted into a diagram see Figure \ref{fig:3x3modulator}.
	\begin{figure}[!ht]
		\centering
		\includegraphics{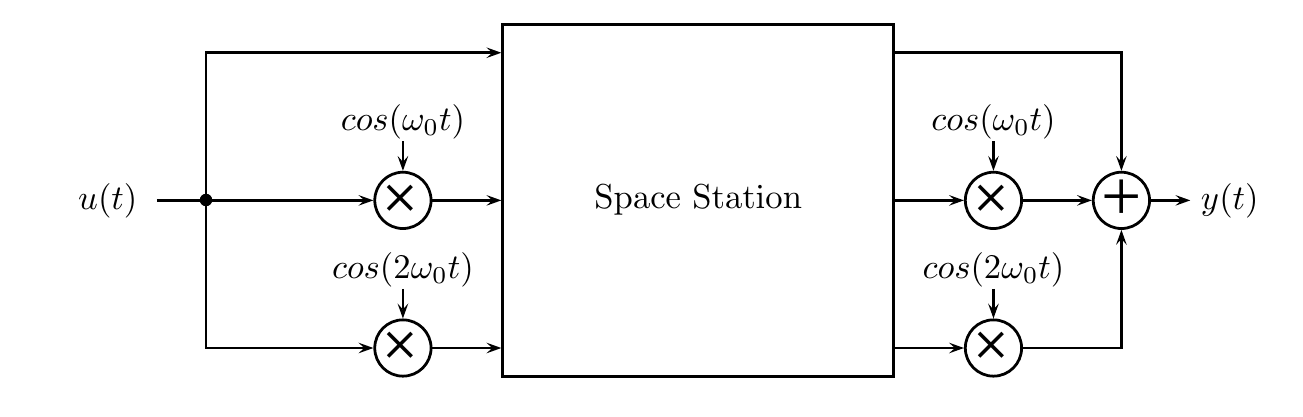}
		\caption{Schematic of the structural model of Component 1r(Russian service module) of the International Space Station with inputs and outputs modulated by local oscillator frequencies $\omega_0$ and $2\omega_0$.}
		\label{fig:3x3modulator}
	\end{figure}

	The system $\ltp$ is already in the Floquet-Fourier form of dimension $n = 270$ and $N = 2$ since the Fourier expansions of $\b(t)$ and $\c(t)$ have 5 nontrivial terms.
	We tried to reduce the model by  \textsf{POD}, using a test function of $u(t) = \sin(19.2875t)$ but each reduced model was unstable for reduced orders greater than or equal to 5. 
	The frequency $\omega = 19.2875$ \caleb{[rad/sec]} was chosen after determining that \caleb{it} excited many of the system harmonics.
	Since the resulting reduced systems were unstable, we omitted the \textsf{POD} simulation results from these comparisons.
	
	For $r = 30$, we run simulations of the full model and reduced model obtained via Algorithm \ref{PropMeth}.
	As shown in Figure \ref{fig:structuraltimesim}, the reduced \textsf{LTP} model is almost indistinguishable from the full \textsf{LTP} model.
	
	\begin{figure}[ht!]
		\centering
	\begin{tikzpicture}
		\begin{axis}[
			width = 0.9\textwidth,
			height = 0.32\textheight,
			xlabel = {Time, $t$, [sec]},
			legend entries = {Full Simulation, H2-Based},
			]
			\addplot+[mark = none,black] table [x=time, y=y_FOM, col sep=comma] {ss_timesim_FOM.dat};
			\addplot+[mark = none,thick,red,densely dashed] table [x=time, y=y_ROM, col sep=comma] {ss_timesim_ROM.dat};
		\end{axis}
	\end{tikzpicture}
		\caption{Time simulation of the modified space station structural model in Figure \ref{fig:3x3modulator}. Full-order ($n = 270$) and reduced-ordre ($r=30$) systems simulated with sinusoidal input, $u(t) = \sin(19.2875t)$.}
		\label{fig:structuraltimesim}
	\end{figure}
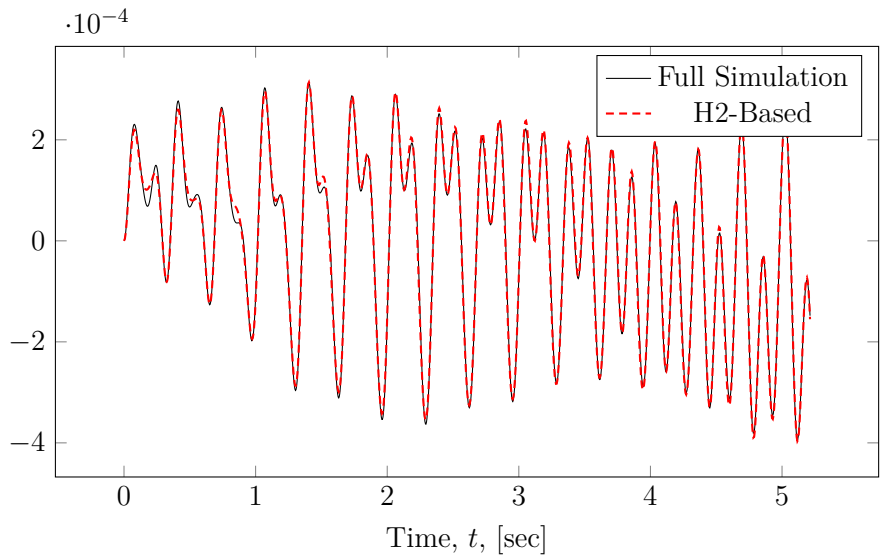
	
	As in the previous example, we construct reduced order models of dimensions $r = 4$ to $r = 80$ using Algorithm \ref{PropMeth} and compare their respective error system $\H2$ error norm $\|\ltp-\widetilde \ltp\|_{\H2}$ in Figure \ref{fig:structuralerrorvr}. The error bound accurately predicts the true error. 

	\begin{figure}[ht!]
		\centering
		\begin{tikzpicture}
			\begin{semilogyaxis}[
				width = 0.9\textwidth,
				height = 0.3\textheight,
				xlabel = {Dimension of Reduced System, $r$},
				ylabel = {$\H2$ Error},
				xmin = 4,
				xmax = 80,
				legend entries = {H2-Based Error,H2-Based Error Bd},
				]
				\addplot+[mark = square,thick,black] table [x=r, y=ltpnorm, col sep=comma] {ss_error_vs_r.dat};
				\addplot+[mark = square,dashed,thick,black] table [x=r, y=ltpnorm_bd, col sep=comma] {ss_error_vs_r.dat};
			\end{semilogyaxis}
		\end{tikzpicture}
		\caption{$\H2$ error and \emph{a posteriori} error bounds for the H2-based model reduction method for the modified space station model in Figure \ref{fig:3x3modulator}.}
		\label{fig:structuralerrorvr}
	\end{figure}
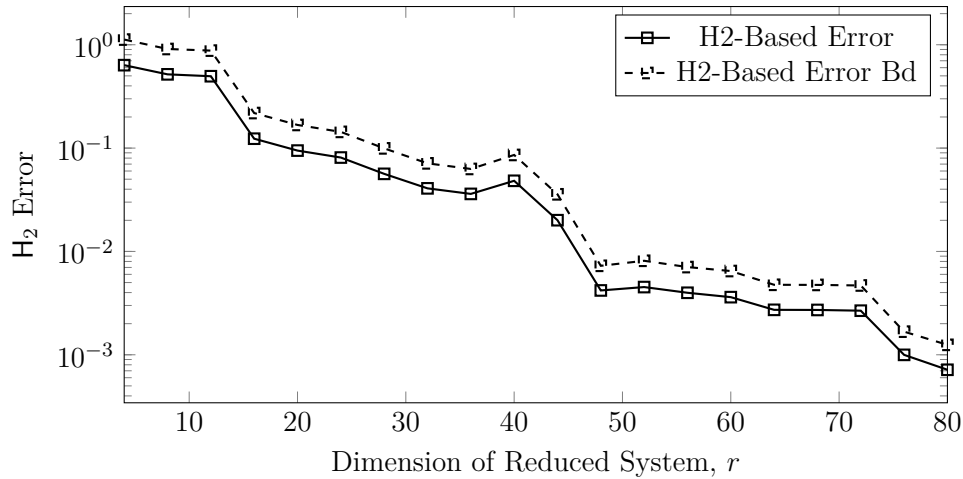
		
\section{Conclusions}

We develop a model reduction scheme for \textsf{LTP} systems by converting the model reduction 
problem into an analogous  \textsf{LTI} problem \caleb{and} employing existing model reduction techniques to the new problem.
Numerical results demonstrate the success of the proposed method.
Moreover, we extend the analysis of the certain notions of $\H2$-approximation established for  \textsf{LTI}  systems to \textsf{LTI}  systems.

In practice, current approaches for computing Floquet transformations do not scale well to large-scale systems, so this aspect may remain a bottleneck for effective model reduction for general \textsf{LTP} systems.  In a variety of circumstances, however, this step is not difficult or problematic.  We do not consider this aspect of the problem in the present work and 
the algorithm we \caleb{propose} here truncates the Floquet-Fourier coefficients, keeping $2N+1$ centered coefficients.
The error introduced by this step becomes arbitrarily small as we increase the number of coefficients conserved, $N\to \infty$.
Note that for any given $N$,  the coefficients that are kept need not be an optimal choice and
investigating what may constitute a better selection of coefficients would benefit any further refinement of this approach.

\section{Acknowledgements}
We thank to Dr.\ Taku Noda and \serkan{Dr.\ Jens Saak} for generously providing us with the MATLAB code used to generate the results in their paper \cite{noda2003harmonic} \serkan{and \cite{lang2015towards}, respectively.}
The work of C.\ Magruder was supported in part by the ExxonMobil Ken Kennedy Institute 2015/2016 High Performance Computing Graduate Fellowship.
The work of  S.\ Gugercin was supported in part by NSF through Grants DMS-1217156 and DMS-1522616, and 
by the Alexander von Humboldt Foundation.   The work of C.\ Beattie was supported in part by NSF through Grant DMS-1217156 and by the Einstein Foundation - Berlin.

\bibliographystyle{nMCM}

\begin{thebibliography}{10}
\newcommand{\noopsort}[1]{}
\newcommand{\printfirst}[2]{#1}
\newcommand{\singleletter}[1]{#1}
\newcommand{\switchargs}[2]{#2#1}
\providecommand{\url}[1]{\normalfont{#1}}
\providecommand{\urlprefix}{Available at }

\bibitem{Gugercin:2008bc}
S. Gugercin, A.C. Antoulas, and C. Beattie, \emph{$\mathcal{H}_2$ model
  reduction for large-scale linear dynamical systems}, SIAM journal on matrix
  analysis and applications 30 (2008), pp. 609--638.

\bibitem{Antoulas:2009tb}
A.C. Antoulas, \emph{Approximation of large-scale dynamical systems}, Vol.~6,
  {SIAM}, 2005.

\bibitem{BenGS15}
P. Benner, S. Gugercin, and K. Willcox, \emph{A survey of projection-based
  model reduction methods for parametric dynamical systems}, SIAM Review 57
  (2015), pp. 483--531.

\bibitem{antoulas2001survey}
A.C. Antoulas, D.C. Sorensen, and S. Gugercin, \emph{A survey of model
  reduction methods for large-scale systems}, Contemporary Mathematics 280
  (2001), pp. 193--220.

\bibitem{sandberg2004balanced}
H. Sandberg and A. Rantzer, \emph{Balanced truncation of linear time-varying
  systems}, Automatic Control, IEEE Transactions on 49 (2004), pp. 217--229.

\bibitem{sandberg2006case}
H. Sandberg, \emph{A case study in model reduction of linear time-varying
  systems}, Automatica 42 (2006), pp. 467--472.

\bibitem{lang2015towards}
N. Lang, J. Saak, and T. Stykel, \emph{Balanced truncation model reduction for
  linear time-varying systems}, Mathematical and Computer Modelling of
  Dynamical Systems 22 (2016), pp. 267--281.

\bibitem{varga1999balancing}
A. Varga, \emph{Balancing related methods for minimal realization of periodic
  systems}, Systems \& Control Letters 36 (1999), pp. 339--349.

\bibitem{varga2000balanced}
A. Varga, \emph{Balanced truncation model reduction of periodic systems}, in
  \emph{Decision and Control, 2000. Proceedings of the 39th IEEE Conference
  on}, Vol.~3, 2000, pp. 2379--2384.

\bibitem{farhood2005model}
M. Farhood, C.L. Beck, and G.E. Dullerud, \emph{Model reduction of periodic
  systems: a lifting approach}, Automatica 41 (2005), pp. 1085--1090.

\bibitem{Chahlaoui2005}
Y. Chahlaoui and P. Van~Dooren, \emph{Model reduction of time-varying systems},
  in \emph{Dimension reduction of large-scale systems}, P. Benner, D.C.
  Sorensen, and V. Mehrmann, eds., Springer,  2005, pp. 131--148.

\bibitem{benner2014low}
P. Benner, M.S. Hossain, and T. Stykel, \emph{Low-rank iterative methods for
  periodic projected {L}yapunov equations and their application in model
  reduction of periodic descriptor systems}, Numerical Algorithms 67 (2014),
  pp. 669--690.

\bibitem{Wereley:1991us}
N.M. Wereley and S.R. Hall, \emph{Linear time periodic systems: transfer
  function, poles, transmission zeroes and directional properties}, in
  \emph{American Control Conference}, 28, 1991, pp. 1179--1184.

\bibitem{Wereley:1990wj}
N.M. Wereley and S.R. Hall, \emph{Frequency response of linear time periodic
  systems}, in \emph{Decision and Control, 1990., Proceedings of the 29th IEEE
  Conference on}, 1990, pp. 3650--3655.

\bibitem{Sandberg:2006tk}
H. Sandberg, \emph{On {F}loquet-{F}ourier Realizations of Linear Time-Periodic
  Impulse Responses}, in \emph{Decision and Control, 2006 45th IEEE Conference
  on}, 2006, pp. 1411--1416.

\bibitem{Sandberg:2004vh}
H. Sandberg, E. M{\"o}llerstedt, \emph{et~al.}, \emph{Frequency-domain analysis
  of linear time-periodic systems}, Automatic Control, IEEE Transactions on 50
  (2005), pp. 1971--1983.

\bibitem{Zhou:2005vs}
J. Zhou and T. Hagiwara, \emph{{Finite-dimensional models in evaluating the
  $\mathcal{H}_2$ norm of continuous-time periodic systems}}, Proc. of IFAC
  2005 World Congress, Prague, Czech Republic  (2005).

\bibitem{Zhou:2004vh}
J. Zhou, T. Hagiwara, and M. Araki, \emph{{Spectral characteristics and
  eigenvalues computation of the harmonic state operators in continuous-time
  periodic systems}}, Systems {\&} Control Letters 53 (2004), pp. 141--155.

\bibitem{Zhou:2002tp}
J. Zhou and T. Hagiwara, \emph{$\mathcal{H}_2$ and $\mathcal{H}_\infty$ norm
  computations of linear continuous-time periodic systems via the skew analysis
  of frequency response operators}, Automatica 38 (2002), pp. 1381--1387.

\bibitem{Zhou:2002wf}
J. Zhou and T. Hagiwara, \emph{{Existence conditions and properties of the
  frequency response operators of continuous-time periodic systems}}, SIAM
  Journal on Control and Optimization 40 (2002), p. 1867.

\bibitem{Zhou:2002tf}
J. Zhou, T. Hagiwara, and M. Araki, \emph{{Stability analysis of
  continuous-time periodic systems via the harmonic analysis}}, IEEE
  Transactions on Automatic Control 47 (2002), pp. 292--298.

\bibitem{Ant2010imr}
A.C. Antoulas, C.A. Beattie, and S. Gugercin, \emph{Interpolatory model
  reduction of large-scale dynamical systems}, in \emph{Efficient Modeling and
  Control of Large-Scale Systems}, Springer, 2010, pp. 3--58.

\bibitem{beattie2012realization}
C. Beattie and S. Gugercin, \emph{Realization-independent
  $\mathcal{H}_2$-approximation}, in \emph{Decision and Control (CDC), 2012
  IEEE 51st Annual Conference on}, 2012, pp. 4953--4958.

\bibitem{flagg2012convergence}
G. Flagg, C. Beattie, and S. Gugercin, \emph{Convergence of the {I}terative
  {R}ational {K}rylov {A}lgorithm}, Systems \& Control Letters 61 (2012), pp.
  688--691.

\bibitem{beattie2009trm}
C. Beattie and S. Gugercin, \emph{A trust region method for optimal
  {$\mathcal{H}_2$} model reduction}, 48th IEEE Conference on Decision and
  Control  (2009).

\bibitem{KRXC08}
A. Kellems, D. Roos, N. Xiao, and S. Cox, \emph{Low-dimensional,
  morphologically accurate models of subthreshold membrane potential}, Journal
  of Computational Neuroscience 27 (2009), pp. 161--176.

\bibitem{borggaard2012model}
J. Borggaard, E. Cliff, and S. Gugercin, \emph{Model reduction for indoor-air
  behavior in control design for energy-efficient buildings}, in \emph{American
  Control Conference (ACC), 2012}, 2012, pp. 2283--2288.

\bibitem{grimshaw1991nonlinear}
R. Grimshaw, \emph{Nonlinear ordinary differential equations}, Vol.~2, CRC
  Press, 1991.

\bibitem{Moore:2005un}
G. Moore, \emph{Floquet theory as a computational tool}, SIAM Journal on
  Numerical Analysis 42 (2005), pp. 2522--2568.

\bibitem{Cai:2001wh}
Z. Cai, Y. Gu, and W. Zhong, \emph{A new approach of computing {F}loquet
  transition matrix}, Computers \& Structures 79 (2001), pp. 631--635.

\bibitem{borwein2005weighted}
D. Borwein and W. Kratz, \emph{Weighted convolution operators on $l_p$},
  Canadian Mathematical Bulletin 48 (2005), pp. 175--179.

\bibitem{zygmund2002trigonometric}
A. Zygmund, \emph{Trigonometric series}, Vol.~1, Cambridge university press,
  2002.

\bibitem{Stykel_Vasilyev}
T. Stykel and A. Vasilyev, \emph{A two-step model reduction approach for
  mechanical systems with moving loads}, J. Comput. Appl. Math. 297 (2016), pp.
  85--97.

\bibitem{magruder2013model}
C.C. Magruder~III, \emph{Model reduction of linear time-periodic dynamical
  systems}, Master's thesis, Virginia Tech,  2013.

\bibitem{MulR76}
C. Mullis and R. Roberts, \emph{Synthesis of minimum roundoff noise fixed point
  digital filters}, IEEE Transactions on Circuits and Systems 23 (1976), pp.
  551--562.

\bibitem{Moo81}
B. Moore, \emph{Principal component analysis in linear systems:
  Controllability, observability, and model reduction}, IEEE Transactions on
  Automatic Control 26 (1981), pp. 17--32.

\bibitem{glover1984all}
K. Glover, \emph{All optimal {Hankel}-norm approximations of linear
  multivariable systems and their $l_\infty$-error bounds}, International
  journal of control 39 (1984), pp. 1115--1193.

\bibitem{kunisch2001galerkin}
K. Kunisch and S. Volkwein, \emph{Galerkin proper orthogonal decomposition
  methods for parabolic problems}, Numerische mathematik 90 (2001), pp.
  117--148.

\bibitem{steih2012space}
K. Steih and K. Urban, \emph{Space-time reduced basis methods for time-periodic
  partial differential equations}, IFAC Proceedings Volumes 45 (2012), pp.
  710--715.

\bibitem{sorensen2003direct}
D.C. Sorensen and Y. Zhou, \emph{Direct methods for matrix {S}ylvester and
  {L}yapunov equations}, Journal of Applied Mathematics 2003 (2003), pp.
  277--303.

\bibitem{noda2003harmonic}
T. Noda, A. Semlyen, and R. Iravani, \emph{Harmonic domain dynamic transfer
  function of a nonlinear time-periodic network}, Power Delivery, IEEE
  Transactions on 18 (2003), pp. 1433--1441.

\end{thebibliography}

\end{document}